\documentclass[a4paper,12pt,reqno]{amsart}
\usepackage{amssymb}
\usepackage[left=1in,right=1in,top=1.1in,bottom=1.1in]{geometry}
\usepackage{graphicx,color}
\usepackage{enumitem}

\theoremstyle{plain}
\newtheorem{theorem}{Theorem}[section]
\newtheorem*{theorem*}{Theorem}
\newtheorem{lemma}[theorem]{Lemma}
\newtheorem{corollary}[theorem]{Corollary}
\newtheorem{proposition}[theorem]{Proposition}

\theoremstyle{definition}
\newtheorem{definition}[theorem]{Definition}
\newtheorem{example}[theorem]{Example}
\newtheorem{remark}[theorem]{Remark}

\newtheorem{assumption}{Assumption}

\numberwithin{equation}{section}

\DeclareMathOperator{\supp}{supp}
\DeclareMathOperator{\diam}{diam}
\DeclareMathOperator{\dist}{dist}

\DeclareMathOperator{\diverg}{div}

\newcommand{\A}{\mathcal{A}}

\newcommand{\X}{\mathfrak{X}}
\newcommand{\dom}{\mathcal{D}}
\newcommand{\pr}{\mathbf{P}}
\newcommand{\ex}{\mathbf{E}}
\newcommand{\R}{\mathbf{R}}
\newcommand{\M}{\mathcal{M}}
\newcommand{\ind}{\mathbf{1}}
\newcommand{\sub}{\subseteq}
\newcommand{\ph}{\varphi}
\newcommand{\ro}{\varrho}
\newcommand{\eps}{\varepsilon}
\newcommand{\pot}{\mathcal{U}}
\newcommand{\potY}{\mathcal{V}}
\newcommand{\norm}[1]{\left\| #1 \right\|}

\newcommand{\set}[1]{\left\{ #1 \right\}}

\newcommand{\expr}[1]{\left( #1 \right)}
\newcommand{\as}{\text{-a.s.}}
\newcommand{\aevery}{\text{-a.e.}}

\newcommand{\itref}[1]{\ref{#1}}

\begin{document}

%
%

\title{Boundary Harnack inequality \linebreak for Markov processes with jumps}
\author{Krzysztof Bogdan}
\author{Takashi Kumagai}
\author{Mateusz Kwa{\'s}nicki}
\address{Krzysztof Bogdan and Mateusz Kwa{\'s}nicki: Institute of Mathematics \\ Polish Academy of Science \\ ul. {\'S}niadeckich 8 \\ 00-958 Warsaw, Poland, and Institute of Mathematics and Computer Science \\ Wroc{\l}aw University of Technology \\ ul. Wybrze{\.z}e Wyspia{\'n}\-skiego 27 \\ 50-370 Wroc{\l}aw, Poland}
\address{Takashi Kumagai: Research Institute for Mathematical Sciences \\ Kyoto University \\ Kyoto 606-8502, Japan}
\email{krzysztof.bogdan@pwr.wroc.pl, kumagai@kurims.kyoto-u.ac.jp, \newline mateusz.kwasnicki@pwr.wroc.pl}
\thanks{Krzysztof Bogdan was supported in part by grant N N201 397137.}
\thanks{Takashi Kumagai was supported by the Grant-in-Aid for Challenging Exploratory Research 24654033.}
\thanks{Mateusz Kwa\'snicki was supported by the Foundation for Polish Science
and by the Polish Ministry of Science and Higher Education grant N N201 373136.}
\date{\today}
\begin{abstract}
We prove a boundary Harnack inequality for jump-type Markov processes on metric measure state spaces, under comparability estimates of the jump kernel and Urysohn-type property of the domain of the generator of the process. The result holds for positive harmonic functions in arbitrary open sets. It applies, e.g., to many subordinate Brownian motions, L\'evy processes with and without continuous part, stable-like and censored stable processes, jump processes on fractals, and rather general Schr\"odinger, drift and jump perturbations of such processes.
\end{abstract}

\maketitle

%
%

\section{Introduction}\label{sec:int}

The boundary Harnack inequality ({BHI}) is a statement about nonnegative functions which are harmonic on an open set and vanish outside the set near a part of its boundary. {BHI} asserts that the functions have a common boundary decay rate. The property requires proper assumptions on the set and the underlying Markov process, ones which secure relatively good communication from near the boundary to the center of the set. By this we mean that the process starting near the boundary visits the center of the set at least as likely as creep far along the boundary before leaving the set.

{BHI} for harmonic functions of the Laplacian $\Delta$ in Lipschitz domains was proved in 1977--78 by B. Dahlberg, A. Ancona and J.-M.~Wu (\cite{MR513885, MR0466593, MR513884}), after a pioneering attempt of J.~Kemper (\cite{MR0293114,MR0422644}). In 1989 R.~Bass and K.~Burdzy proposed an alternative probabilistic proof based on elementary properties of the Brownian motion (\cite{MR1042338}). The resulting `box method' was then applied to more general domains, including H{\"o}lder domains of order $r>1/2$, and to more general second order elliptic operators (\cite{MR1127476, MR1277760}). {BHI} trivially fails for disconnected sets, and counterexamples for H\"older domains with $r {<} 1/2$ are given in \cite{MR1277760}. In 2001--09, H.~Aikawa studied {BHI} for classical harmonic functions in connection to the Carleson estimate and under exterior capacity conditions (\cite{MR1800526, MR2464701, MR2484620}).   

Moving on to nonlocal operators and jump-type Markov processes, in 1997 K. Bogdan proved  {BHI} for the fractional Laplacian $\Delta^{\alpha/2}$ (and the isotropic $\alpha$-stable L\'evy process) for $0<\alpha<2$ and Lipschitz sets (\cite{MR1438304}). In 1999 R.~Song and J.-M.~Wu extended the results to the so-called fat sets (\cite{MR1719233}), and in 2007 K.~Bogdan, T.~Kulczycki and M.~Kwa\'snicki proved {BHI} for $\Delta^{\alpha/2}$ in arbitrary, in particular disconnected, open sets (\cite{MR2365478}). In 2008 P.~Kim, R.~Song and Z.~Vondra\v{c}ek proved {BHI} for subordinate Brownian motions in fat sets (\cite{MR2513121}) and in 2011 extended it to a general class of isotropic L\'evy processes and arbitrary domains (\cite{MR2994122}). Quite recently, BHI for $\Delta+\Delta^{\alpha/2}$ was established by Z.-Q.~Chen, P.~Kim, R.~Song and Z.~Vondra\v{c}ek \cite{0908.1559}. We also like to mention BHI for censored processes \cite{MR2006232, 2007arXiv0705.1614G} by K.~Bogdan, K.~Burdzy, Z.-Q.~Chen and Q.~Guan, and fractal jump processes \cite{MR2792590, MR2261965} by K.~Kaleta, M.~Kwa\'snicki and A.~St\'os.

Generally speaking, 
{BHI} is more a topological issue for diffusion processes, and more a measure-theoretic issue for jump-type Markov processes, which may transport from near the boundary to the center of the set by direct jumps. However, \cite{MR1438304,MR2365478} show in a special setting that such jumps determine the asymptotics of harmonic functions only at those boundary points where the set is rather thin, while at other boundary points the main contribution to the asymptotics comes from gradual `excursions' away from the boundary.

We recall that BHI in particular applies to and may yield an approximate factorization of the Green function. This line of research was completed for Lipschitz domains in 2000 by K.~Bogdan (\cite{MR1741527}) for $\Delta$ and in 2002 by T.~Jakubowski (\cite{MR1991120}) for $\Delta^{\alpha/2}$. It is now a well-established technique (\cite{MR2160104}) and extensions were proved, e.g., for  subordinate Brownian motions by P.~Kim, R.~Song and Z.~Vondra\v{c}ek (\cite{Kim17012012}). We should note that so far the technique is typically restricted to Lipschitz or fat sets. Furthermore, for smooth sets, e.g. $C^{1,1}$ sets, the approximate factorization is usually more explicit. This is so because for smooth sets the decay rate in BHI can often be explicitly expressed in terms of the distance to the {boundary} of the set. The first complete results in this direction were given for $\Delta$ in 1986 by Z.~Zhao (\cite{MR842803}) and for $\Delta^{\alpha/2}$ in 1997  by T.~Kulczycki (\cite{MR1490808}) and in 1998 by Z.-Q.~Chen and R.~Song (\cite{MR1654824}). The estimates are now extended to subordinate Brownian motions, and the renewal function of the subordinator is used in the corresponding formulations (\cite{Kim17012012}). Accordingly, the Green function of smooth sets enjoys approximate factorization for rather general isotropic L\'evy processes (\cite{MR1942325, Kim17012012}). We expect further progress in this direction with applications to perturbation theory via the so-called 3G theorems, and to nonlinear partial differential equations (\cite{springerlink:10.1007/s11118-011-9237-x, MR2160104, MR2294482}). We should also mention estimates and approximate factorization of the Dirichlet heat kernels, which are intensively studied at present. The estimates depend on BHI (\cite{MR2722789}), and reflect the fundamental decay rate in BHI (\cite{MR2677618,MR2807275}).

BHI tends to self-improve and may lead to the existence of the boundary limits of ratios of nonnegative harmonic functions, thanks to oscillation reduction (\cite{MR1042338, MR1438304, MR2365478, MR716504}). The oscillation reduction technique is rather straightforward for local operators. It is more challenging for non-local operators, as it involves subtraction of harmonic functions, which destroys global nonnegativity. The technique requires a certain scale invariance, or uniformity of BHI, and works, e.g., for $\Delta$ in Lipschitz domains  (\cite{MR1042338}) and for $\Delta^{\alpha/2}$ in arbitrary domains (\cite{MR2365478}). We should remark that H\"older continuity of harmonic functions is a similar phenomenon, related to the usual Harnack inequality, and that BHI extends the usual Harnack inequality if, e.g., constant functions are harmonic. H\"older continuity of harmonic functions is crucial in the theory of partial differential equations \cite{MR2735074,MR1918242}, and the existence of limits of ratios of nonnegative harmonic functions leads to the construction of the Martin kernel and to representation of nonnegative harmonic functions (\cite{MR2075671,MR2365478}).

The above summary indicate further directions of research resulting from our development. The main goal of this article is to study the following boundary Harnack inequality. In Section~\ref{sec:amr} we specify notation and assumptions which validate the estimate. 
\medskip
\begin{enumerate}[label={\rm (BHI)},ref={\rm (BHI)}]
\item\label{it:bhi}
Let $x_0 \in \X$, $0 < r < R < R_0$, and let $D \sub B(x_0, R)$ be open. Suppose that nonnegative functions $f, g$ on $\X$ are regular harmonic in $D$ with respect to the process $X_t$, and vanish in $B(x_0, R) \setminus D$. There is $c_{\eqref{eq:ubhi}} = c_{\eqref{eq:ubhi}}(x_0, r, R)$ such that
\begin{align}
\label{eq:ubhi}
 f(x)g(y) & \le c_{\eqref{eq:ubhi}} \, f(y)g(x) \, , & x, y \in B(x_0, r).
\end{align}
\end{enumerate}
\medskip
Here $X_t$ is a Hunt process, having a metric measure space $\X$ as the state space, and $R_0 \in (0, \infty]$ is a localization radius (discussed in Section~\ref{sec:amr}). Also, a nonnegative function $f$ is said to be regular harmonic in $D$ with respect to $X_t$ if 
\begin{align}\label{eq:rharmonicity}
 f(x) & = \ex_x f(X_{\tau_D}) , & x \in D ,
\end{align}
where $\tau_D$ is the time of the first exit of $X_t$ from $D$. To facilitate cross-referencing, in \eqref{eq:ubhi} and later on we let $c_{(i)}$ denote the constant in the displayed formula $(i)$. By $c$ or $c_i$ we denote secondary (temporary) constants in a lemma or a section, and $c = c(a, \dots, z)$, or simply $c(a, \dots, z)$, means a constant $c$ that may be so chosen to depend only on $a, \dots, z$. Throughout the article, all constants are positive.

The present work started with an attempt to obtain \emph{bounded} kernels which reproduce harmonic functions. We were motivated by the so-called regularization of the Poisson kernel for $\Delta^{\alpha/2}$ (\cite{MR1671973}, \cite[Lemma~6]{MR2365478}), which is crucial for the Carleson estimate and BHI for $\Delta^{\alpha/2}$. In the present paper we construct kernels obtained by gradually stopping the Markov process with a specific multiplicative functional before the process approaches the boundary. The construction is the main technical ingredient of our work, and is presented in Section~\ref{sec:reg}. The argument is intrinsically probabilistic and relies on delicate analysis on the path space. At the beginning of Section~\ref{sec:reg} the reader will also find a short informal presentation of the construction. Section~\ref{sec:amr} gives assumptions and auxiliary results. The boundary Harnack inequality (Theorem~\ref{th:bhi}), and the so-called local supremum estimate (Theorem~\ref{th:regexit}) are presented in Section~\ref{sec:bhi}, but the proof of Theorem~\ref{th:regexit} is deferred to Section~\ref{sec:reg}. In Section~\ref{sec:ex} we verify in various settings the scale-invariance of BHI, discuss the relevance of our main assumptions from Section~\ref{sec:amr}, and present many applications, including subordinate Brownian motions, L\'evy processes with or without continuous part, stable-like and censored processes, Schr\"odinger, gradient and jump perturbations, processes on fractals and more.

\section{Assumptions and Preliminaries}
\label{sec:amr}

Let $(\X, d, m)$ be a metric measure space such that all bounded closed sets are compact and $m$ has full support. Let $B(x, r) = \{y \in \X: d(x,y) < r\}$, where $x \in \X$ and $r > 0$. All sets, functions and measures considered in this paper are Borel. Let $R_0 \in (0, \infty]$ (the localization radius) be such that $\X \setminus B(x, 2 r) \ne \varnothing$ for all $x \in \X$ and all $r < R_0$. Let $\X \cup \set{\partial}$ be the one-point compactification of $\X$ (if $\X$ is compact, then we add $\partial$ as an isolated point). Without much mention we extend functions $f$ on $\X$ to $\X \cup \set{\partial}$ by letting $f(\partial) = 0$. In particular, we write $f \in C_0(\X)$ if $f$ is a continuous real-valued function on $\X \cup \set{\partial}$ and $f(\partial) = 0$. If furthermore $f$ has compact support in $\X$, then we write $f \in C_c(\X)$. For a kernel $k(x, dy)$ on $\X$ (\cite{MR939365}) we let $k f(x) = \int f(y) k(x, dy)$, provided the integral makes sense, i.e., $f$ is (measurable and) either nonnegative or absolutely integrable. Similarly, for a kernel density function $k(x, y) \ge 0$, we let $k(x, E) = \int_E k(x, y) m(dy)$ and $k(E, y) = \int_E k(x, y) m(dx)$ for $E \sub \X$. 

Let $(X_t, \zeta, \M_t, \pr_x)$ be a Hunt process with state space $\X$ (see, e.g.,~\cite[I.9]{MR0264757} or~\cite[3.23]{MR0193671}). Here $X_t$ are the random variables, $\M_t$ is the usual right-continuous filtration, $\pr_x$ is the distribution of the process starting from $x \in \X$, and $\ex_x$ is the corresponding expectation. The random variable $\zeta \in (0, \infty]$ is the lifetime of $X_t$, so that $X_t = \partial$ for $t \ge \zeta$. This should be kept in mind when interpreting \eqref{eq:rharmonicity} above, \eqref{eq:semigr} below, etc. The transition operators of $X_t$ are defined by
\begin{align}\label{eq:semigr}
 T_t f(x) & = \ex_x f(X_t), & t \ge 0, \, x \in \X,
\end{align}
whenever the expectation makes sense. We assume that the semigroup $T_t$ is Feller and strong Feller, i.e., for $t > 0$, $T_t$ maps bounded functions into continuous ones and $C_0(\X)$ into $C_0(\X)$. The Feller generator $\A$ of $X_t$ is defined on the set $\dom(\A)$ of all those $f \in C_0(\X)$ for which the limit
\begin{align*}
 \A f(x) & = \lim_{t \searrow 0} \frac{T_t f(x) - f(x)}{t}
\end{align*}
exists uniformly in $x \in \X$. The $\alpha$-potential operator,
\begin{align*}
 \pot_\alpha f(x) & = \ex_x \int_0^\infty f(X_t) e^{-\alpha t} dt = \int_0^\infty e^{-\alpha t} T_t f(x) dt , & \alpha \ge 0, \, x \in \X ,
\end{align*}
is defined whenever the expectation makes sense. We let $\pot = \pot_0$, the potential operator. The kernels of $T_t$, $\pot_\alpha$ and $\pot$ are denoted by $T_t(x, dy)$, $\pot_\alpha(x, dy)$ and $\pot(x, dy)$, respectively.

Recall that a function $f \ge 0$ is called $\alpha$-excessive (with respect to $T_t$) if for all $x \in \X$, $e^{-\alpha t} T_t f(x) \le f(x)$ for $t > 0$, and $e^{-\alpha t} T_t f(x) \to f(x)$ as $t \to 0^+$. When $\alpha = 0$, we simply say that $f$ is excessive.

We enforce a number of conditions, namely Assumptions~\ref{assu:dual}, \ref{assu:gen}, \ref{assu:levy} and~\ref{assu:green} below.
We start with a duality assumption, which builds on our discussion of $X_t$.

\begin{assumption}
\label{assu:dual}
There are Hunt processes $X_t$ and $\hat{X}_t$ which are dual with respect to the measure $m$ (see~\cite[VI.1]{MR0264757} or~\cite[13.1]{MR2152573}). The transition semigroups of $X_t$ and $\hat{X}_t$ are both Feller and strong Feller. Every semi-polar set of $X_t$ is polar.
\end{assumption}

In what follows, objects pertaining to $\hat{X}_t$ are distinguished in notation from those for $X_t$ by adding a hat over the corresponding symbol. For example, $\hat{T}_t$ and $\hat{\pot}_\alpha$ denote the transition and $\alpha$-potential operators of $\hat{X}_t$. The first sentence of Assumption~\ref{assu:dual} means that for all $\alpha > 0$, there are functions $\pot_\alpha(x, y) = \hat{\pot}_\alpha(y, x)$ such that
\begin{align*}
 \pot_\alpha f(x) & = \int_{\X} \pot_\alpha(x, y) f(y) m(dy), & \hat{\pot}_\alpha f(x) & = \int_{\X} \hat{\pot}_\alpha(x, y) f(y) m(dy)
\end{align*}
for all $f \ge 0$ and $x \in \X$, and such that $x \mapsto \pot_\alpha(x, y)$ is $\alpha$-excessive with respect to $T_t$, and $y \mapsto \pot_\alpha(x, y)$ is $\alpha$-excessive with respect to $\hat{T}_t$ (that is, $\alpha$-co-excessive). The $\alpha$-potential kernel $\pot_\alpha(x, y)$ is unique (see \cite[Theorem 13.2]{MR2152573} or remarks after \cite[Proposition~VI.1.3]{MR0264757}). 

The condition in Assumption~\ref{assu:dual} that semi-polar sets are polar is also known as \emph{Hunt's hypothesis} (H). Most notably, it implies that the process $X_t$ never hits irregular points, see, e.g.,~\cite[I.11 and II.3]{MR0264757} or~\cite[Chapter~3]{MR2152573}. The $\alpha$-potential kernel is non-increasing in $\alpha > 0$, and hence the potential kernel $\pot(x, y) = \lim_{t \to 0^+} \pot_\alpha(x, y) \in [0, \infty]$ is well-defined.

We consider an open set $D\subset\X$ and the time of the first exit from $D$ for $X_t$ and $\hat{X}_t$,
\begin{align*}
 \tau_D = \inf \{ t \ge 0 : X_t \notin D \} && \mbox{and} && \hat{\tau}_D = \inf\{ t \ge 0 : \hat{X}_t \notin D \}.
\end{align*}
We define the processes killed at $\tau_D$,
\begin{align*}
 X^D_t = \begin{cases}
  X_t , & \text{if $t < \tau_D$,} \\
  \partial , & \text{if $t \ge \tau_D$,} 
 \end{cases}
 && \mbox{and} && \hat{X}^D_t = 
  \begin{cases}
  \hat{X}_t , & \text{if $t < \hat{\tau}_D$,} \\
  \partial , & \text{if $t \ge \hat{\tau}_D$.}
 \end{cases}
\end{align*} 
We let $T^D_t(x, dy)$ and $\hat{T}^D_t(x, dy)$ be their transition kernels. By~\cite[Remark 13.26]{MR2152573}, $X^D_t$ and $\hat{X}^D_t$ are dual processes with state space $D$. Indeed, for each $x \in D$, $\pr_x\as$ the process $X_t$ only hits regular points of $\X \setminus D$ when it exits $D$. In the nomenclature of~\cite[13.6]{MR2152573}, this means that the left-entrance time and the hitting time of $\X \setminus D$ are equal $\pr^x\as$ for every $x \in D$. In particular, the potential kernel $G_D(x, y)$ of $X^D_t$ exists and is unique, although in general it may be infinite (\cite[pp.~256--257]{MR0264757}). $G_D(x, y)$ is called the Green function for $X_t$ on $D$, and it defines the Green operator $G_D$,
\begin{align*}
 G_D f(x) & = \int_{\X} f(y) G_D(x,y) m(dy) = \ex_x \int_0^{\tau_D} f(X_t) dt , & x \in \X , f \ge 0.
\end{align*}
Note that $\pot(x, y) = G_{\X}(x, y)$. When $X_t$ is symmetric (self-dual) with respect to $m$, then Assumption~\ref{assu:dual} is equivalent to the existence of the $\alpha$-potential kernel $\pot_\alpha(x, y)$ for $X_t$, since then Hunt's hypothesis (H) is automatically satisfied, see~\cite{MR2152573}.

The following Urysohn regularity hypothesis plays a crucial role in our paper, providing enough `smooth' functions on $\X$ to approximate indicator functions of compact sets.

\begin{assumption}
\label{assu:gen}
There is a linear subspace $\dom$ of $\dom(\A) \cap \dom(\hat{\A})$ satisfying the following condition. If $K$ is compact, $D$ is open, and $K \sub D \sub \X$, then there is $f \in \dom$ such that $f(x) = 1$ for $x \in K$, $f(x) = 0$ for $x \in \X \setminus D$, $0 \le f(x) \le 1$ for $x \in \X$, and the boundary of the set $\{x \, : \, f(x) > 0\}$ has measure $m$ zero. We let 
\begin{align}
 \ro(K, D) & = \inf_{f} \sup_{x \in \X} \max(\A f(x), \hat{\A} f(x)),\label{eq:defro}
\end{align}
where the infimum is taken over all such functions $f$. 
\end{assumption}

Thus, nonnegative functions in $\dom(\A) \cap \dom(\hat{\A})$ separate the compact set $K$ from the closed set $\X\setminus D$: there is a Urysohn (bump) function for $K$ and $\X\setminus D$ in the domains. 
Since the supremum in~\eqref{eq:defro} is finite for any $f \in \dom$ and the infimum is taken over a nonempty set, $\varrho(K, D)$ is always finite.

Note that constant functions are not in $\dom(\A)$ nor $\dom(\hat{\A})$ unless $\X$ is compact. In the Euclidean case $\X = \R^d$, $\dom$ can often be taken as the class $C_c^\infty(\R^d)$ of compactly supported smooth functions. The existence of $\dom$ is problematic if $\X$ is more general. However, for the Sierpi{\'n}ski triangle and some other self-similar (p.c.f.) fractals, $\dom$ can be constructed by using the concept of splines on fractals (\cite{MR2792590, MR1765920}). Also, a class of smooth indicator functions was recently constructed in~\cite{MR2465816} for heat kernels satisfying upper sub-Gaussian estimates on $\X$. Further discussion is given in Section~\ref{sec:ex} and Appendix~\ref{sec:app}. Here we note that Assumption~\ref{assu:gen} implies that the jumps of $X_t$ are subject to the following identity, which we call the L\'evy system formula for $X_t$,
\begin{equation}
\label{eq:Ls}
 \ex_x \sum_{s \in [0, t]} f(s, X_{s-}, X_s) = \ex_x \int_0^t \int_\X f(s, X_{s-}, z) \nu(X_{s-}, dz) ds.
\end{equation}
Here $f : [0, \infty) \times \X \times \X \to [0,\infty]$, $f(x, x) = 0$ for all $x \in \X$, and $\nu$ is a kernel on $\X$ (satisfying $\nu(x, \set{x}) = 0$ for all $x \in \X$), called the L\'evy kernel of $X_t$, see~\cite{MR0343375, MR958914, MR0185675}. For more general Markov processes, $ds$ in~\eqref{eq:Ls} is superseded by the differential of a perfect, continuous additive functional, and~\eqref{eq:Ls} defines $\nu(x, \cdot)$ only up to a set of zero potential, that is, for $m$-almost every $x \in \X$. By inspecting the construction in~\cite{MR0343375, MR958914}, and using Assumption~\ref{assu:gen}, one proves in a similar way as in~\cite[Section~5]{MR542886} that the L\'evy kernel $\nu$ satisfies
\begin{align}
\label{eq:lk}
 \nu f(x) & = \lim_{t \searrow 0} \frac{T_t f(x)}{t} \, , && f \in C_c(\X), \, x \in \X \setminus \supp f .
\end{align}
This formula, as opposed to~\eqref{eq:Ls}, defines $\nu(x, dy)$ for all $x \in \X$. With only one exception, to be discussed momentarily, we use~\eqref{eq:lk} and not~\eqref{eq:Ls}, hence we take~\eqref{eq:lk} as the definition of $\nu$. It is easy to see that~\eqref{eq:lk} indeed defines $\nu(x, dy)$: if $f \in \dom(\A)$ and $x \in \X \setminus \supp f$, then $\nu f(x) = \A f(x)$. By Assumption~\ref{assu:gen}, the mapping $f \mapsto \nu f(x)$ is a densely defined, nonnegative linear functional on $C_c(\X \setminus \{x\})$, hence it corresponds to a nonnegative Radon measure $\nu(x, dy)$ on $\X \setminus \{x\}$. As usual, we let $\nu(x, \{x\}) = 0$. {The L{\'e}vy kernel $\hat{\nu}(y, dx)$ for $\hat{X}_t$ is defined in a similar manner. By duality, $\nu(x, dy) m(dx) = \hat{\nu}(y, dx) m(dy)$.}

As an application of \eqref{eq:Ls} we consider the martingale
\begin{align*}
 t \mapsto \sum_{s \in [0, t]} f(s, X_{s-}, X_s) - \int_0^t \int_\X f(s, X_{s-}, z) \nu(X_{s-}, dz) ds,
\end{align*}
where $f(s, y, z) = \ind_A(s) \ind_E(y) \ind_F(z)$. We stop the martingale at $\tau_D$ and we see that
\begin{equation}
\label{eq:iw2}
 \pr_x(\tau_D \in dt, X_{\tau_D-} \in dy, X_{\tau_D} \in dz) = dt \, T^D_t(x, dy) \nu(y, dz),
\end{equation}
on $(0, \infty) \times D \times (\X \setminus D)$. A similar result was first proved in~\cite{MR0142153}. For this reason we refer to~\eqref{eq:iw2} as the Ikeda-Watanabe formula (see also~\eqref{eq:iw} and \eqref{eq:iwe} below). Integrating \eqref{eq:iw2} against $dt$ and $dy$ we obtain
\begin{align}
\label{eq:iwe}
 \pr_x(X_{\tau_D-} \ne X_{\tau_D}, X_{\tau_D} \in E) & = \int_D G_D(x,dy)\nu(y,E), && x \in D, \, E \subset \X \setminus D.
\end{align}

For $x_0 \in \X$ and $0 < r < R$, we consider the open and closed balls $B(x_0, r) = \set{x \in \X : d(x_0, x) < r}$ and $\overline{B}(x_0, r) = \set{x \in \X : d(x_0, x) \le r}$, and the annular regions $A(x_0, r, R) = \set{x \in \X : r < d(x_0, x) < R}$ and $\overline{A}(x_0, r, R) = \set{x \in \X : r \le d(x_0, x) \le R}$. Note that $\overline{B(x_0, r)}$, the closure of $B(x_0,r)$, may be a proper subset of $\overline{B}(x_0, r)$. 

Recall that $R_0$ denotes the localization radius of $\X$. The following assumption is our main condition for the boundary Harnack inequality. It asserts a relative constancy of the density of the L{\'e}vy kernel. This is a natural condition, as seen in Example~\ref{ex:truncated}.

\begin{assumption}
\label{assu:levy}
The L{\'e}vy kernels of the processes $X_t$ and $\hat{X}_t$ have the form $\nu(x, y) m(dy)$ and $\hat{\nu}(x, y) m(dy)$ respectively, where $\nu(x, y) = \hat{\nu}(y, x) > 0$ for all $x, y \in \X$, $x \ne y$. For every $x_0 \in \X$, $0 < r < R < R_0$, $x \in B(x_0, r)$ and $y \in \X \setminus B(x_0, R)$,
\begin{align}
\label{eq:nu}
 c_{\eqref{eq:nu}}^{-1}{\nu(x_0, y)} \le {\nu(x, y)} \le c_{\eqref{eq:nu}}{\nu(x_0, y)} , &&
 c_{\eqref{eq:nu}}^{-1}{\hat\nu(x_0, y)} \le {\hat\nu(x, y)} \le c_{\eqref{eq:nu}}{\hat\nu(x_0, y)} , 
\end{align}
with $c_{\eqref{eq:nu}} = c_{\eqref{eq:nu}}(x_0, r, R)$. 
\end{assumption}

It follows directly from Assumption~\ref{assu:levy} that for $x_0 \in \X$ and $0 < r < R$, 
\begin{equation}
\label{eq:nu2}
 c_{\eqref{eq:nu2}}(x_0, r, R) = \inf_{y \in \overline{A}(x_0, r, R)} \min(\nu(x_0, y), \hat{\nu}(x_0, y)) > 0 
\end{equation} 
where $\overline{A}(x_0, r, R) = \set{x \in \X : r \le d(x_0, x) \le R}$. (Here we do not require that $R < R_0$.) Indeed, we may cover $\overline{A}(x_0, r, R)$ by a finite family of balls $B(y_i, r/2)$, where $y_i \in \overline{A}(x_0, r, R)$. For $y \in B(y_i, r/2)$, $\nu(x_0, y)$ is comparable with $\nu(x_0, y_i)$, and $\hat{\nu}(x_0, y)$ is comparable with $\hat{\nu}(x_0, y_i)$.

\begin{proposition}
\label{prop:tau}
If $x_0 \in \X$ and $0 < r < R_0$, then
\begin{equation}
\label{eq:tau}
 c_{\eqref{eq:tau}}(x_0, r) = \sup_{x \in B(x_0, r)} \max(\ex_x \tau_{B(x_0, r)}, \hat{\ex}_x \hat{\tau}_{B(x_0, r)}) < \infty .
\end{equation}
\end{proposition}
\begin{proof}
Let $B = B(x_0, r)$, $R \in (r, R_0)$, $x, y \in B$ and $F(t) = \pr_x(\tau_B > t)$. By the definition of $R_0$, $m(\X \setminus B(x_0, R)) > 0$. This and~\eqref{eq:nu} yield $\nu(y, \X \setminus B) \ge \nu(y, \X \setminus B(x_0, R)) \ge (c_{\eqref{eq:nu}}(x_0, r, R))^{-1} \nu(x_0, \X \setminus B(x_0, R)) = c$. By the Ikeda-Watanabe formula~\eqref{eq:iw2},
\begin{align*}
 -F'(t) & = \frac{\pr_x(\tau_B \in dt)}{dt} \ge \frac{\pr_x(\tau_B \in dt, X_{\tau_B-} \ne X_{\tau_B}, X_{\tau_B} \in \X \setminus B)}{dt} \\
 & = \int_\X \nu(y, \X \setminus B) T^B_t(x, dy) \ge c \int_\X T^B_t(x, dy) = c F(t) .
\end{align*}
Hence $\pr_x(\tau_B > t) \le e^{-c t}$. If follows that $\ex_x \tau_B \le 1/c$. 
Considering an analogous argument for $\hat{\ex}_x \hat{\tau}_B$, we see that we may take
\begin{equation*}
 c_{\eqref{eq:tau}}(x_0, r) = \inf_{R \in (r, R_0)} \max\expr{\frac{c_{\eqref{eq:nu}}(x_0, r, R)}{\nu(x_0, \X \setminus B(x_0, R))}, \frac{c_{\eqref{eq:nu}}(x_0, r, R)}{\hat{\nu}(x_0, \X \setminus B(x_0, R))}} . \qedhere
\end{equation*}
\end{proof}

In particular, if $0 < R < R_0$ and $D \sub B(x_0, R)$, then the Green function $G_D(x, y)$ exists (see the discussion following Assumption~\ref{assu:dual}), and for each $x \in \X$ it is finite for all $y$ in $\X$ less a polar set. We need to assume slightly more. The following condition may be viewed as a weak version of Harnack's inequality.

\begin{assumption}
\label{assu:green}
If $x_0 \in \X$, $0 < r < p < R < R_0$ and $B = B(x_0, R)$, then 
\begin{equation}
\label{eq:green}
 c_{\eqref{eq:green}}(x_0, r, p, R) = \sup_{x \in B(x_0, r)} \sup_{y \in \X \setminus B(x_0, p)} \max(G_B(x, y), \hat{G}_B(x, y)) < \infty .
\end{equation}
\end{assumption}

Assumptions~\ref{assu:dual}, \ref{assu:gen}, \ref{assu:levy} and~\ref{assu:green} are tacitly assumed throughout the entire paper. We recall them explicitly only in the statements of BHI and local maximum estimate.

When saying that a statement holds for almost every point of $\X$, we refer to the measure $m$. The following technical result is a simple generalization of \cite[Proposition~II.3.2]{MR0264757}.

\begin{proposition}
\label{prop:aa}
Suppose that $Y_t$ is a standard Markov process such that for every $x \in \X$ and $\alpha > 0$, the $\alpha$-potential kernel $\potY_\alpha(x, dy)$ of $Y_t$ is absolutely continuous with respect to $m(dy)$. Suppose that function $f$ is excessive for the transition semigroup of $Y_t$, and $f$ is not identically infinite. If function $g$ is continuous and $f(x) \le g(x)$ for \emph{almost every} $x \in B(x_0, r)$, then $f(x) \le g(x)$ for \emph{every} $x \in B(x_0, r)$.
\end{proposition}

\begin{proof}
Let $A = \{x \in B(x_0, r) : f(x) > g(x)\}$. Then $m(A) = 0$, so that $A$ is of zero potential for $Y$. Hence $B(x_0, r) \setminus A$ is finely dense in $B(x_0, r)$. Since $f - g$ is finely continuous, we have $f(x) \le g(x)$ for all $x \in B(x_0, r)$, as desired. (See e.g.~\cite{MR0264757, MR2152573} for the notion of fine topology and fine continuity of excessive functions.)
\end{proof}

If $X_t$ is transient, \eqref{eq:green} often holds even when $G_B$ is replaced by $G_\X = \pot$. In the recurrent case, we can use estimates of $\pot_\alpha$, as follows.

\begin{proposition}
\label{prop:green:pot}
If $x_0 \in \X$, $0 < r < p < R < R_0$, $\alpha > 0$,
\[
 c_1(x_0, r, p, \alpha) = \sup_{x \in B(x_0, r)} \sup_{y \in \X \setminus B(x_0, p)} \max(\pot_\alpha(x, y), \hat{\pot}_\alpha(x, y)) < \infty ,
\]
and $T_t(x, dy) \le c_2(t) m(dy)$ for all $x, y \in \X$, $t > 0$, then in \eqref{eq:green} we may let
\[
 c_{\eqref{eq:green}}(x_0, r, p, R) = \inf_{\alpha, t > 0} \expr{e^{\alpha t} c_1(x_0, r, p, \alpha) + c_2(t) c_{\eqref{eq:tau}}(x_0, R)} .
\]
\end{proposition}

\begin{proof}
Denote $B = B(x_0, R)$. If $x \in B(x_0,r)$, $t_0 > 0$ and $E \sub B \setminus B(x_0, p)$, then
\begin{align*}
 G_B \ind_E(x) & = \int_0^\infty T^B_t \ind_E(x) dt \\
 & \le e^{\alpha t_0} \int_0^{t_0} e^{-\alpha t} T^B_t \ind_E(x) dt + \int_0^\infty T^B_s (T^B_{t_0} \ind_E)(x) ds \\
 & \le e^{\alpha t_0} \int_0^\infty e^{-\alpha t} T_t \ind_E(x) dt + \expr{\sup_{y \in B} T^B_{t_0} \ind_E(y)} \int_0^\infty T^B_s \ind(y) ds \\
 & \le e^{\alpha t_0} \pot_\alpha \ind_E(x) + \expr{\sup_{y \in B} T_{t_0} \ind_E(y)} G_B \ind(x) \\
 & \le (e^{\alpha t_0} c_1 + c_2 G_B \ind(x)) |E|,
\end{align*}
where $c_1=c_1(x_0, r, p, \alpha)$ and $c_2=c_2(t_0)$. If $y \in B \setminus B(x_0, p)$, then by Proposition~\ref{prop:aa}, 
$G_B(x, y) \le e^{\alpha t_0} c_1 + c_2 G_B \ind(x)$. By Proposition~\ref{prop:tau}, $G_B \ind(x) = \ex_x \tau_B \le c_\eqref{eq:tau}(x_0, R)$. The estimate of $\hat{G}_B(x, y)$ is similar.
\end{proof}

We use the standard notation $\ex_x(Z; E)= \ex_x (Z \ind_E)$. Recall that all functions $f$ on $\X$ are automatically extended to $\X \cup \{\partial\}$ by letting $f(\partial) = 0$. In particular, we understand that $\A f(\partial) = 0$ for all $f \in \dom(\A)$, and $\ex_x \A f(X_\tau) = \ex_x(\A f(X_\tau); \tau < \zeta)$.

The following formula obtained by Dynkin (see \cite[formula~(5.8)]{MR0193671}) plays an important role. If $\tau$ is a Markov time, $\ex_x \tau < \infty$ and $f \in \dom(\A)$, then
\begin{align}
\label{eq:d}
 \ex_x f(X_\tau) & = f(x) + \ex_x \int_0^\tau \A f(X_t) dt, && x \in \X .
\end{align}
If $D \sub B(x_0, R_0)$, $f \in \dom(\A)$ is supported in $\X \setminus D$ and $X_t \in D$ $\pr_y\as$ for $t < \tau$ and $x \in \X$, then
\begin{equation}
\label{eq:iw}
\begin{split}
 \ex_x f(X_\tau) & = \ex_x \int_0^\tau \expr{\int_{\X} \nu(X_t, y) f(y) m(dy)} dt \\
 & = \int_{\X} \ex_x\expr{\int_0^\tau\nu(X_t, y) dt} f(y) m(dy) .
\end{split}
\end{equation}
We note that \eqref{eq:iw} extends to nonnegative functions $f$ on $\X$ which vanish on $\overline{D}$. Indeed, both sides of~\eqref{eq:iw} define nonnegative functionals of $f \in C_0(\X \setminus \overline{D})$, and hence also nonnegative Radon measures on $\X \setminus \overline{D}$. 
By~\eqref{eq:iw}, the two functionals coincide on $\dom \cap C_0(\X \setminus \overline{D})$, and this set is dense in $C_0(\X \setminus \overline{D})$ by the Urysohn regularity hypothesis. This proves that the corresponding measures are equal.
We also note that one cannot in general relax the condition that $f = 0$ on $\overline{D}$. Indeed, even if $m(\partial D) = 0$, $X_\tau$ may hit $\partial D$ with positive probability.

Recall that a function $f \ge 0$ on $\X$ is called regular harmonic in an open set $D \sub \X$ if $f(x) = \ex_x f(X(\tau_D))$ for all $x \in \X$. Here a typical example is $x\mapsto \ex_x \int_0^\infty g(X_t)dt$ if $g\ge 0$ vanishes on $D$.
By the strong Markov property we then have
$f(x) = \ex_x f(X_\tau)$ for all stopping times $\tau\le \tau_D$. Accordingly, we call $f \ge 0$ regular subharmonic in $D$ (for $X_t$), if 
$f(x) \le \ex_x f(X_\tau)$ for all stopping times $\tau\le \tau_D$ and $x\in \X$.
Here a typical example is a regular harmonic function raised to a power $p\ge 1$.
We like to recall that $f \ge 0$ is called harmonic in $D$, if $f(x) = \ex_x f(X(\tau_U))$ for all open and bounded $U$ such that $\overline{U} \sub D$, and all $x \in U$. This condition is satisfied, e.g., by the Green function $G_D(\cdot, y)$ in $D \setminus \{y\}$, and it is weaker than regular harmonicity. In this work however, only the notion of regular harmonicity is used. For further discussion, we refer to \cite{2009arXiv0912.3290C,Hansen2008, MR0193671, MR0150831}.

%
%

\section{Boundary Harnack inequality}
\label{sec:bhi}

Recall that Assumptions~\ref{assu:dual}, \ref{assu:gen}, \ref{assu:levy} and~\ref{assu:green} are in force throughout the entire paper. Some results, however, hold in greater generality. For example, the following Lemma~\ref{lem:escape} relies solely on Assumption~\ref{assu:gen} and \eqref{eq:tau}, and it remains true also when $X_t$ is a diffusion process. Also, Lemma~\ref{lem:factorization} and Corollary~\ref{corr:ubhi} require Assumptions~\ref{assu:gen} and~\ref{assu:levy} but not
~\ref{assu:dual} or \ref{assu:green}.

\begin{lemma}
\label{lem:escape}
If $x_0 \in \X$ and $0 < r < R < \tilde{R}<\infty$, then for all $D \sub B(x_0, R)$ we have
\begin{align}
\label{eq:escape}
 \pr_x(X_{\tau_D} \in \overline{A}(x_0, R, \tilde{R})) & \le c_{\eqref{eq:escape}} \ex_x \tau_D, && x \in B(x_0, r) \cap D,
\end{align}
where $c_{\eqref{eq:escape}} = c_{\eqref{eq:escape}}(x_0, r, R, \tilde{R}) = \inf_{\tilde{r} > \tilde{R}} \ro(\overline{A}(x_0, R, \tilde{R}), A(x_0, r, \tilde{r}))$.
\end{lemma}

\begin{proof}
We fix an auxiliary number $\tilde{r} > \tilde{R}$ and $x \in B(x_0, r)$.
Let $f \in \dom$ be a bump function from Assumption~\ref{assu:gen} for the compact set $\overline{A}(x_0, R, \tilde{R})$ and the open set $A(x_0, r, \tilde{r})$. Thus, $f \in \dom(\A)$, $f(x) = 0$, $f(y) = 1$ for $y \in \overline{A}(x_0, R, \tilde{R})$ and $0 \le f(y) \le 1$ for all $y \in \X$. By Dynkin's formula~\eqref{eq:d} we have
\begin{align*}
 \pr_x(X_{\tau_D} \in \overline{A}(x_0, R, \tilde{R})) & \le \ex_x(f(X_{\tau_D})) - f(x) 
 = G_D (\A f)(x) \le \smash{G_D \ind(x) \sup_{y \in \X} \A f(y)} .
\end{align*}
Since $G_D \ind(x) = \ex_x \tau_D$, the proof is complete.
\end{proof}

We write $f \approx c g$ if $c^{-1} g \le f \le c g$. We will now clarify the relation between BHI and local supremum estimate.

\begin{lemma}
\label{lem:factorization}
The following conditions are equivalent:
\begin{enumerate}[label={\rm (\alph{*})},ref={\rm (\alph{*})}]
\item\label{it:factorization:a}
If $x_0 \in \X$, $0 < r < R < R_0$, $D \sub B(x_0, R)$ is open, $f$ is nonnegative, regular harmonic in $D$ and vanishes in $B(x_0, R) \setminus D$, then
\begin{align}
\label{eq:upper}
 f(x) & \le c_{\eqref{eq:upper}} \int_{\X \setminus B(x_0, r)} f(y) \nu(x_0, y) m(dy)
\end{align}
for $x \in B(x_0, r) \cap D$, where $c_{\eqref{eq:upper}} = c_{\eqref{eq:upper}}(x_0, r, R)$.
\item\label{it:factorization:b}
If $x_0 \in \X$, $0 < r < p < q < R < R_0$, $D \sub B(x_0, R)$ is open, $f$ is nonnegative, regular harmonic in $D$ and vanishes in $B(x_0, R) \setminus D$, then
\begin{align}
\label{eq:factorization}
 f(x) & \approx c_{\eqref{eq:factorization}} \ex_x(\tau_{D \cap B(x_0, p)}) \int_{\X \setminus B(x_0, q)} f(y) \nu(x_0, y) m(dy)
\end{align}
for $x \in B(x_0, r) \cap D$, where $c_{\eqref{eq:factorization}} = c_{\eqref{eq:factorization}}(x_0, r, p, q, R)$.
\end{enumerate}
In fact, if~\itref{it:factorization:a} holds, then we may let
\begin{align*}
 c_{\eqref{eq:factorization}}(x_0, r, p, q, R) & = c_{\eqref{eq:escape}}(x_0, r, p, q) c_{\eqref{eq:upper}}(x_0, q, R) + c_{\eqref{eq:nu}}(x_0, p, q),
\end{align*}
and if \itref{it:factorization:b} holds, then we may let
\begin{align*}
 c_{\eqref{eq:upper}}(x_0, r, R) & = \inf_{\genfrac{}{}{0pt}{}{p,q}{r < p < q < R}} c_{\eqref{eq:factorization}}(x_0, r, p, q, R) c_{\eqref{eq:tau}}(x_0, R) .
\end{align*}
\end{lemma}

\begin{proof}
Since $\X\setminus B(x_0,q) \sub \X\setminus B(x_0,r)$ and $\ex_x(\tau_{D \cap B(x_0, p)}) \le \ex_x(\tau_{B(x_0, R)}) \le c_{\eqref{eq:tau}}(x_0, R)$, we see that \itref{it:factorization:b} implies \itref{it:factorization:a} with $c_{\eqref{eq:upper}} = c_{\eqref{eq:factorization}}(x_0, r, p, q, R) c_{\eqref{eq:tau}}(x_0, R)$. Below we prove the converse. Let \itref{it:factorization:a} hold, and $U = D \cap B(x_0, p)$. We have
\begin{align}
\label{eq:decomposition}
 f(x) & = \ex_x(f(X_{\tau_U}) ; X_{\tau_U} \in \overline{B}(x_0, q)) + \ex_x(f(X_{\tau_U}) ; X_{\tau_U} \in \X \setminus \overline{B}(x_0, q)) .
\end{align}
Denote the terms on the right hand side by $I$ and $J$, respectively. By~\eqref{eq:escape} and~\eqref{eq:upper},
\begin{equation}
\label{eq:decomposition:I}
\begin{split}
 0 \le I & \le \pr_x(X_{\tau_U} \in \overline{A}(x_0, p, q)) \sup_{y \in B(x_0, q)} f(y) \\
 & \le c_{\eqref{eq:escape}} c_{\eqref{eq:upper}} \ex_x \tau_U \int_{\X \setminus B(x_0, q)} f(y) \nu(x_0, y) m(dy) ,
\end{split}
\end{equation}
with $c_{\eqref{eq:escape}}(x_0, r, p, q)$ and $c_{\eqref{eq:upper}}(x_0, q, R)$. For $J$, the Ikeda-Watanabe formula~\eqref{eq:iw} yields
\begin{equation}
\label{eq:decomposition:J}
\begin{split}
 J & = \int_{\X \setminus \overline{B}(x_0, q)} \expr{\int_U G_U(x, z) \nu(z, y) f(y) m(dz)} m(dy) \\
 & \approx c_{\eqref{eq:nu}} \int_{\X \setminus \overline{B}(x_0, q)} \expr{\int_U G_U(x, z) \nu(x_0, y) f(y) m(dz)} m(dy) \\
 & = c_{\eqref{eq:nu}} \ex_x \tau_U \int_{\X \setminus \overline{B}(x_0, q)} \nu(x_0, y) f(y) m(dy) ,
\end{split}
\end{equation}
with constant $c_{\eqref{eq:nu}}(x_0, p, q)$. Formula~\eqref{eq:factorization} follows, as we have $c_{\eqref{eq:escape}} c_{\eqref{eq:upper}} + c_{\eqref{eq:nu}}$ in the upper bound and $1 / c_{\eqref{eq:nu}}$ in the lower bound.
\end{proof}

We like to remark that 
BHI boils down to the approximate factorization~\eqref{eq:factorization} of $f(x) = \pr_x(X(\tau_D) \in E)$. We also note that $\pr_x(X(\tau_D) \in E) \approx \nu(x_0, E) \ex_x \tau_D$, if $E$ is far from $B(x_0, R)$, since then $\nu(z, E) \approx \nu(x_0, E)$ in \eqref{eq:iwe}. However, $\nu(z, E)$ in \eqref{eq:iwe} is quite singular and much larger than $\nu(x_0, E)$ if both $z$ and $E$ are close to $\partial B(x_0, R)$. Our main task is to prove that the contribution to~\eqref{eq:iwe} from such points $z$ is compensated by the relatively small time spent there by $X^D_t$ when starting at $x \in D$. In fact, we wish to control~\eqref{eq:iwe} by an integral free from singularities (i.e.~\eqref{eq:upper}), if $x$ and $E$ are not too close.

By substituting~\eqref{eq:factorization} into~\eqref{eq:ubhi}, we obtain the following result.

\begin{corollary}
\label{corr:ubhi}
The conditions \itref{it:factorization:a}, \itref{it:factorization:b} of Lemma~\ref{lem:factorization} imply \itref{it:bhi} with
\begin{align*}
 c_{\eqref{eq:ubhi}}(x_0, r, R) & = \inf_{\genfrac{}{}{0pt}{}{p,q}{r < p < q < R}} (c_{\eqref{eq:factorization}}(x_0, r, p, q, R))^4 . \qed
\end{align*}
\end{corollary}

The main technical result of the paper is the following \emph{local supremum estimate} for sub-harmonic functions, which is of independent interest. The result is proved in Section~\ref{sec:reg}.

\begin{theorem}
\label{th:regexit}
Suppose that Assumptions~\ref{assu:dual}, \ref{assu:gen}, \ref{assu:levy} and~\ref{assu:green} hold true. Let $x_0 \in \X$ and $0 < r < q < R < R_0$, where $R_0$ is the localization radius from Assumptions~\ref{assu:levy} and~\ref{assu:green}. Let function $f$ be nonnegative on $\X$ and regular subharmonic with respect to $X_t$ in $B(x_0, R)$. Then
\begin{align}
\label{eq:regexit1}
 f(x) & \le \int_{\X \setminus B(x_0, q)} f(y) \pi_{x_0, r, q, R}(y) m(dy), && x \in B(x_0, r) ,
\end{align}
where
\begin{align}
\label{eq:regexit2}
 \pi_{x_0, r, q, R}(y) & = \begin{cases} c_{\eqref{eq:regexit3}} \ro & \text{for\/ $y \in B(x_0, R) \setminus B(x_0, q)$,} \\ 2 c_{\eqref{eq:regexit3}} \min(\ro, \hat{\nu}(y, B(x_0, R))) & \text{for\/ $y \in \X \setminus B(x_0, R)$,} \end{cases}
\end{align}
$\ro = \ro(\overline{B}(x_0, q), B(x_0, R))$ (see Assumption~\ref{assu:gen}), and
\begin{align}
\label{eq:regexit3}
 c_{\eqref{eq:regexit3}}(x_0, r, q, R) & = \inf_{p \in (r, q)} \expr{c_{\eqref{eq:green}}(x_0, r, p, R) + \frac{c_{\eqref{eq:tau}}(x_0, R) (c_{\eqref{eq:nu}}(x_0, p, q))^2}{m(B(x_0, p))}} .
\end{align}
\end{theorem}

Theorem~\ref{th:regexit} (to be proved in the next section) and Corollary~\ref{corr:ubhi} lead to BHI. We note that no regularity of the open set $D$ is assumed.

\begin{theorem}
\label{th:bhi}
If assumptions~\ref{assu:dual}, \ref{assu:gen}, \ref{assu:levy} and~\ref{assu:green} are satisfied, then~\itref{it:bhi} holds true with
\begin{align}
\label{eq:ubhic1}
 c_{\eqref{eq:ubhi}}(x_0, r, R) =
 \makebox[2.5em][c]{$\displaystyle \inf_{\genfrac{}{}{0pt}{}{p,q,\tilde{r}}{r < p < q < R < \tilde{r}}}$}
 \expr{\ro(\overline{A}(x_0, p, q), A(x_0, r, \tilde{r})) c_{\eqref{eq:ubhic2}}(x_0, q, R) + c_{\eqref{eq:nu}}(x_0, p, q)}^4 \! ,
\end{align}
\begin{align}
\label{eq:ubhic2}
\begin{split}
 c_{\eqref{eq:ubhic2}}(x_0, q, R) & =
 \makebox[2.5em][c]{$\displaystyle \inf_{\genfrac{}{}{0pt}{}{\tilde{q},\tilde{R}}{q < \tilde{q} < R < \tilde{R}}}$}
 2 c_{\eqref{eq:regexit3}}(x_0, q, \tilde{q}, R) \times \\
 & \quad \times \max \expr{\frac{\ro(\overline{B}(x_0, \tilde{q}), B(x_0, R))}{c_{\eqref{eq:nu2}}(x_0, \tilde{q}, \tilde{R})}, c_{\eqref{eq:nu}}(x_0, R, \tilde{R}) m(B(x_0, R))} .
\end{split}
\end{align}
\end{theorem}

\begin{proof}
We only need to prove condition~\itref{it:factorization:a} of Lemma~\ref{lem:factorization} 
with $c_{\eqref{eq:upper}}$ equal to $c_{\eqref{eq:ubhic2}}=c_{\eqref{eq:ubhic2}}(x_0, r, R)$ given above. By~\eqref{eq:regexit1} and~\eqref{eq:regexit2} of Theorem~\ref{th:regexit}, it suffices to prove that 
$
\inf_{q \in (r, R)}\sup_{y\in \X\setminus B(x_0,q)} 
\pi_{x_0, r, q, R}(y)/\nu(x_0, y) \le c_{\eqref{eq:ubhic2}}$. 
For $y \in \overline{A}(x_0, q, \tilde{R})$ we have
\begin{align*}
 \pi_{x_0, r, q, R}(y) & \le 2 c_{\eqref{eq:regexit3}} \ro \le \frac{2 c_{\eqref{eq:regexit3}} \ro}{c_{\eqref{eq:nu2}}} \, \nu(x_0, y) ,
\end{align*}
with $c_{\eqref{eq:regexit3}} = c_{\eqref{eq:regexit3}}(x_0, r, q, R)$, $\ro = \ro(\overline{B}(x_0, q), B(x_0, R))$ and $c_{\eqref{eq:nu2}} = c_{\eqref{eq:nu2}}(x_0, q, \tilde{R})$.
If $y \in \X \setminus \overline{B}(x_0, \tilde{R})$, then
\begin{align*}
 \pi_{x_0, r, q, R}(y) & \le 2 c_{\eqref{eq:regexit3}} \hat{\nu}(y, B(x_0, R)) \le 2 c_{\eqref{eq:regexit3}} c_{\eqref{eq:nu}} m(B(x_0, R)) \nu(x_0, y) ,
\end{align*}
with $c_{\eqref{eq:regexit3}}$ as above and $c_{\eqref{eq:nu}} = c_{\eqref{eq:nu}}(x_0, R, \tilde{R})$. The proof is complete.
\end{proof}

\begin{remark}
\label{rem:si}
\itref{it:bhi} is said to be \emph{scale-invariant} if $c_{\eqref{eq:ubhi}}$ may be so chosen to depend on $r$ and $R$ only through the ratio $r/R$. In some applications, the property plays a crucial role, see, e.g., \cite{MR1127476, MR2365478}. If $X_t$ admits \emph{stable-like scaling}, then $c_{\eqref{eq:ubhi}}$ given by~\eqref{eq:ubhic1} is scale-invariant indeed, as explained in Section~\ref{sec:ex} (see Theorem~\ref{th:sibhi}).
\end{remark}

\begin{remark}
The constant $c_{\eqref{eq:ubhi}}$ in Theorem~\ref{th:bhi} depends only on basic characteristics of $X_t$. Accordingly, in Section~\ref{sec:ex} it is shown that BHI is stable under small perturbations.
\end{remark}

\begin{remark}
BHI applies in particular to hitting probabilities: if $0 < r < R < R_0$, $x, y \in B(x_0, r) \cap D$ and $E_1, E_2 \sub \X \setminus B(x_0, R)$, then
\begin{align*}
 {\pr_x(X_{\tau_D} \in E_1)}{\pr_y(X_{\tau_D} \in E_2)} & \le c_{\eqref{eq:ubhi}} \, {\pr_y(X_{\tau_D} \in E_1)}{\pr_x(X_{\tau_D} \in E_2)} .
\end{align*}
\end{remark}

\begin{remark}
BHI implies the usual Harnack inequality if, e.g., constants are harmonic.
\end{remark}

The approach to BHI via approximate factorization was applied to isotropic stable processes in~\cite{MR2365478}, to stable-like subordinate diffusion on the Sierpi\'nski gasket in~\cite{MR2792590}, and to a wide class of isotropic L\'evy processes in~\cite{MR2994122}. In all these papers, the taming of the intensity of jumps near the boundary was a crucial step. This parallels the connection of the Carleson estimate and BHI in the classical potential theory, see Section~\ref{sec:int}.

%
%

\section{Regularization of the exit distribution}
\label{sec:reg}

In this section we prove Theorem~\ref{th:regexit}. The proof is rather technical, so we begin with a few words of introduction and an intuitive description of the idea of the proof.

In~\cite[Lemma~6]{MR2365478}, an analogue of Theorem~\ref{th:regexit} was obtained for the isotropic $\alpha$-stable L{\'e}vy processes by averaging harmonic measure of the ball against the variable radius of the ball.  The procedure yields a kernel with no singularities and a 
mean value property for 
harmonic functions. In the setting of \cite{MR2365478} the boundedness of the kernel follows from the explicit formula and bounds for the harmonic measure of a ball.  A similar argument is classical for 
harmonic functions of the Laplacian and the Brownian motion. For more general processes $X_t$ this approach is problematic: while the Ikeda-Watanabe formula gives precise bounds for the harmonic measure far from the ball, satisfactory estimates near the boundary of the ball require exact decay rate of the Green function, which is generally unavailable. In fact, resolved cases indicate that sharp estimates of the Green function are equivalent to BHI (\cite{MR1741527}), hence not easier to obtain. Below we use a different method to mollify the harmonic measure.

Recall that the harmonic measure of $B$ is the distribution of $X(\tau_B)$. It may be interpreted as the mass lost by a particle moving along the trajectory of $X_t$, when it is killed at the moment $\tau_B$. In the present paper we let the particle lose the mass \emph{gradually} before time $\tau_B$, with intensity $\psi(X_t)$ for a suitable function $\psi \ge 0$ sharply increasing at $\partial B$. The resulting distribution of the lost mass defines a kernel with a mean value property for harmonic functions, and it is less singular than the distribution of $X(\tau_B)$. 

\medskip

Throughout this section, we fix $x_0 \in \X$ and four numbers $0 < r < p < q < R < R_0$, where $R_0$ is defined in Assumptions~\ref{assu:levy} and~\ref{assu:green}. For the compact set $\overline{B}(x_0, q)$ and the open set $B(x_0, R)$ we consider the bump function $\ph$ provided by Assumption~\ref{assu:gen}. We let
\begin{equation}
\label{eq:ddelta}
 \delta = \sup_{x \in \X} \max(\A \ph(x), \hat{\A} \ph(x)),
\end{equation}
and 
\begin{equation}
\label{eq:dV}
 V = \set{x \in \X : \ph(x) > 0}.
\end{equation}
We have $V \sub B(x_0, R)$, see Figure~\ref{fig:1}. By Assumption~\ref{assu:gen}, $m(\partial V) = 0$. Note that $\A \ph(x) \le 0$ and $\hat{\A} \ph(x) \le 0$ if $x \in B(x_0, q)$, and $\delta$ can be arbitrarily close to $\ro(\overline{B}(x_0, q), B(x_0, R))$.

\begin{figure}
\centering
\def\svgwidth{0.5\textwidth}
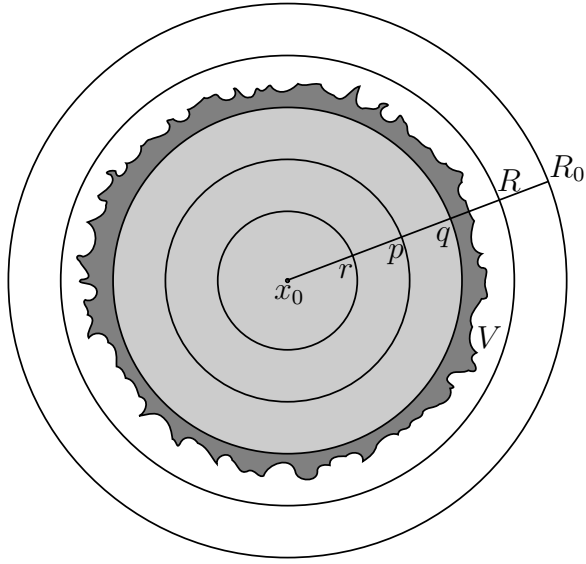
\caption{Notation for Section~\ref{sec:reg}.}
\label{fig:1}
\end{figure}

We consider a function $\psi: \X \cup \{ \partial \} \to [0, \infty]$ continuous in the extended sense and such that $\psi(x) = \infty$ for $x \in (\X \setminus V) \cup \set{\partial}$, and $\psi(x) < \infty$ when $x \in V$. Let
\begin{align}
\label{eq:m}
 A_t & = \lim_{\eps \searrow 0} \int_0^{t + \eps} \psi(X_s) ds , && t \ge 0 .
\end{align}
We see that $A_t$ is a right-continuous, strong Markov, nonnegative (possibly infinite) additive functional, and $A_t = \infty$ for $t \ge \zeta$. We define the right-continuous multiplicative functional 
\begin{align*}
 M_t = e^{-A_t}.
\end{align*}
For $a \in [0, \infty]$, we let $\tau_a$ be the first time when $A_t \ge a$. In particular, $\tau_\infty$ is the time when $A_t$ becomes infinite. Note that $A_t$ and $M_t$ are continuous except perhaps at the single (random) moment $\tau_\infty$ when $A_t$ becomes infinite and the left limit $A(\tau_\infty-)$ is finite. Since $A_t$ is finite for $t < \tau_V$, we have $\tau_\infty \ge \tau_V$. If $\psi$ grows sufficiently fast near $\partial V$, then in fact $\tau_\infty = \tau_V$, as we shall see momentarily.

\begin{lemma}
\label{lem:boundary1}
If $c_1, c_2 > 0$ are such that $\psi(x) \ge c_1 (\ph(x))^{-1} - c_2$ for all $x \in V$, then $A(\tau_V) = \infty$ and $M(\tau_V) = 0$ $\pr_x\as$ for every $x \in \X$. In particular, $\tau_V = \tau_\infty$.
\end{lemma}

\begin{proof}
We first assume  that $x \in \X \setminus V$. In this case it suffices to prove that $A_0 = \infty$. Since $\A \ph(y) \le \delta$ for all $y \in \X$, and $\ph(x) = 0$, from Dynkin's formula for the (deterministic) time $s$ it follows that $\ex_x(\ph(X_s)) \le \delta s$ for all $s > 0$. By the Schwarz inequality, 
\begin{align*}
 \expr{\int_\eps^t \frac{1}{s} \, ds}^2 & \le \expr{\int_\eps^t \frac{\ph(X_s)}{s^2} \, ds} \expr{\int_\eps^t \frac{1}{\ph(X_s)} \, ds},
\end{align*}
where $0 < \eps < t$. Here we use the conventions $1 / 0 = \infty$ and $0 \cdot \infty = \infty$.
Thus,
\begin{align*}
 \ex_x {\left(\expr{\int_\eps^t \frac{1}{\ph(X_s)} \, ds}^{-1}\right)} & \le \expr{\int_\eps^t \frac{1}{s} \, ds}^{-2} \ex_x \expr{\int_\eps^t \frac{\ph(X_s)}{s^2} \, ds} \\
 & \le \expr{\int_\eps^t \frac{1}{s} \, ds}^{-2} \int_\eps^t \frac{\delta}{s} \, ds = \frac{\delta}{\log(t / \eps)} \, ,
\end{align*}
with the convention $1 / \infty = 0$. Hence,
\begin{equation}
\label{eq:exactest}
\begin{split}
 \ex_x \expr{\frac{1}{A_t + c_2 t}} & \le \ex_x {\left(\expr{\int_\eps^t (\psi(X_s) + c_2) ds}^{-1}\right)} \\
 & \le \ex_x {\left(\expr{\int_\eps^t \frac{c_1}{\ph(X_s)} \, ds}^{-1}\right)} \le \frac{\delta}{c_1 \log(t / \eps)} \, .
\end{split}
\end{equation}
By taking $\eps \searrow 0$, we obtain
\[
 \ex_x \expr{\frac{1}{A_t + c_2 t}} = 0 .
\]
It follows that $A_t = \infty$ $\pr_x\as$ We conclude that $A_0 = \infty$ and $M_0 = 0$ $\pr_x\as$, as desired.

When $x \in V$, the result in the statement of the lemma follows from the strong Markov property. 
Indeed, by the definition~\eqref{eq:m} of $A_t$, $A(\tau_V) = A(\tau_V-) + (A_0 \circ \vartheta_{\tau_V})$, where $\vartheta_{\tau_V}$ is the shift operator on the underlying probability space, which shifts sample paths of $X_t$ by the random time $\tau_V$, and $A(\tau_V-)$ denotes the left limit of $A_t$ at $t = \tau_V$. Hence, $M(\tau_V) = M(\tau_V-) \cdot (M_0 \circ \vartheta_{\tau_V})$. Furthermore, $X(\tau_V)\in \X \setminus V$ $\pr_x\as$, so by the first part of the proof, we have $\ex_{X(\tau_V)}(M_0) = 0$ $\pr_x\as$ Thus,
\begin{align*}
 \ex_x M_{\tau_V} & = \ex_x(M_{\tau_V-} \ex_{X(\tau_V)}(M_0)) = 0,
\end{align*}
which implies that $M(\tau_V) = 0$ $\pr_x\as$ and $A(\tau_V) = \infty$ $\pr_x\as$.\end{proof}

From now on we only consider the case when the assumptions of Lemma~\ref{lem:boundary1} are satisfied, and $c_1$, $c_2$ are reserved for the constants in the condition $\psi(x) \ge c_1 (\ph(x))^{-1} - c_2$. By the definition and right-continuity of paths of $X_t$, $A_t$ and $M_t$ are monotone right-differentiable continuous functions of $t$ on $[0, \tau_V)$, with derivatives $\psi(X_t)$ and $-\psi(X_t) M_t$, respectively. 

Let $\eps_a(\cdot)$ be the Dirac measure at $a$. Lemma~\ref{lem:boundary1} yields the following result.

\begin{corollary}
\label{corr:gpsipsi}
We have $-d M_t = \psi(X_t) M_t dt + M(\tau_V-) \eps_{\tau_V}(dt)$ $\pr_x\as$ In particular,
\begin{align}
\label{eq:gpsipsi}
-\ex_x \int_{[0, \tau)} f(X_t) d M_t 
= \ex_x \expr{\int_0^\tau f(X_t) \psi(X_t) M_t dt} + \ex_x \expr{M_{\tau_V-} f(X_{\tau_V}) ; \tau > \tau_V}
\end{align}
for any measurable random time $\tau$ and nonnegative or bounded function $f$.
\qed
\end{corollary}

We emphasize that if $M_t$ has a jump at $\tau$, in which case we must have $\tau = \tau_V$, then the jump does not contribute to the Lebesgue-Stieltjes integral $\int_{[0, \tau)} f(X_t) d M_t$ in~\eqref{eq:gpsipsi}. The same remark applies to~\eqref{eq:psid} below.

Recall that $\tau_a = \inf \set{t \ge 0 : A_t \ge a}$. Note that $\tau_a$ are Markov times for $X_t$, $a \mapsto \tau_a$ is the left-continuous inverse of $t \mapsto A_t$, and the events $\{t < \tau_a\}$ and $\{A_t < a\}$ are equal. We have $A(\tau_a) = a$ unless $\tau_a = \tau_V$, and, clearly, $\tau_a \le \tau_\infty = \tau_V$.

The following may be considered as an extension of Dynkin's formula.

\begin{lemma}
\label{lem:psigen}
For $f \in \dom(\A)$, Markov time $\tau$, and $x \in V$, we have
\begin{align}
\label{eq:psid}
 \ex_x \int_0^\tau \A f(X_t) M_t dt & = \ex_x (f(X_\tau) M_{\tau-}) - f(x) - \ex_x \int_{[0,\tau)} f(X_t) d M_t .
\end{align}
If $g = (\A - \psi) f$ and $\tau \le \tau_V$, then
\begin{align}
\label{eq:psigen}
 \ex_x \int_0^\tau g(X_t) M_t dt & = \ex_x(f(X_\tau) M_{\tau-}) - f(x) .
\end{align}
In fact, \eqref{eq:psid} holds for every strong Markov right-continuous multiplicative functional $M_t$.
\end{lemma}

\begin{proof}
Since $\int_{A_t}^\infty e^{-a} da = M_t$ and $\{t < \tau_a\} = \{A_t < a\}$, by Fubini,
\begin{align*}
 \ex_x \int_0^\tau \A f(X_t) M_t dt & = \ex_x \int_0^\tau \A f(X_t) \expr{\int_0^\infty \ind_{(0, \tau_a)}(t) e^{-a} da} dt \\
 & = \int_0^\infty \expr{\ex_x \int_0^{\min(\tau, \tau_a)} \A f(X_t) dt} e^{-a} da .
\end{align*}
Since $\min(\tau, \tau_a)$ is a Markov time for $X_t$, we can apply Dynkin's formula. It follows that
\begin{align*}
 \ex_x \int_0^{\min(\tau, \tau_a)} \A f(X_t) dt & = \ex_x(f(X_{\min(\tau, \tau_a)})) - f(x) .
\end{align*}
By Fubini and the substitution $\tau_a = t$, $a = A_t$, $e^{-a} = M_t$,
\begin{align*}
 \ex_x \int_0^\tau \A f(X_t) M_t dt & = \int_0^\infty \expr{\ex_x(f(X_{\min(\tau, \tau_a)})) - f(x)} e^{-a} da \\
 & = \ex_x\expr{\int_0^\infty f(X_{\min(\tau, \tau_a)}) e^{-a} da} - f(x) \\
 & = -\ex_x \expr{\int_{[0, \infty)} f(X_{\min(\tau, t)}) d M_t} - f(x).
\end{align*}
We emphasize that the last equality holds true also if $\tau = \tau_V$ with positive probability. We see that \eqref{eq:psid} holds. By~\eqref{eq:gpsipsi} we obtain~\eqref{eq:psigen}.
\end{proof}

The functional $M_t$ is a Feynman-Kac functional, interpreted as the diminishing mass of a particle started at $x \in \X$. We shall estimate the kernel $\pi_\psi(x, dy)$, defined as the expected amount of mass left by the particle at $dy$. Namely, for any nonnegative or bounded $f$ we define
\begin{align}
\label{eq:pipsi2}
 \pi_\psi f(x) & = -\ex_x \int_{[0, \infty)} f(X_t) d M_t , && x \in \X .
\end{align}
Note that $\pi_\psi f(x) = f(x)$ for $x \in \X \setminus V$. By the substitution $\tau_a = t$, $a = A_t$, $e^{-a} = M_t$ and Fubini, we obtain that
\begin{align}
\label{eq:pipsi1}
 \pi_\psi f(x) & = \ex_x \expr{\int_0^\infty f(X_{\tau_a}) e^{-a} da} = \int_0^\infty \ex_x(f(X_{\tau_a})) e^{-a} da .
\end{align}
The potential kernel $G_\psi(x, dy)$ of the functional $M_t$ will play an important role. Namely, for any nonnegative or bounded $f$ we let
\begin{equation}
\label{eq:gpsi}
 G_\psi f(x) = \ex_x \int_0^\infty f(X_t) M_t dt = \ex_x \int_0^\infty \expr{\int_0^{\tau_a} f(X_t) dt} e^{-a} da .
\end{equation} 
In the second equality above, the identities $M_t = \int_{A_t}^\infty e^{-a} da$ and $\set{t < \tau_a} = \set{A_t < a}$ were used together with Fubini, as in the proof of Lemma~\ref{lem:psigen}. We note that $G_\psi(x,dy)$ measures the expected time spent by the process $X_t$ at $dy$, weighted by the decreasing mass of $X_t$ (compare with the similar role of $G_V(x, y)m(dy)$). There is a semigroup of operators $T^\psi_t f(x) = \ex_x(f(X_t) M_t)$ associated with the multiplicative functional $M_t$. Furthermore, $T^\psi_t$ are transition operators of a Markov process $X^\psi_t$, the \emph{subprocess} of $X_t$ corresponding to $M_t$. With the definitions of~\cite{MR0264757}, $M_t$ is a strong Markov right-continuous multiplicative functional and $V$ is the set of permanent points for $M_t$. Therefore, $X^\psi_t$ is a standard Markov process with state space $V$, see~\cite[III.3.12, III.3.13 and the discussion after III.3.17]{MR0264757}. (From~\eqref{eq:exactest} and~\cite[Proposition~III.5.9]{MR0264757} it follows that $M_t$ is an exact multiplicative functional. Furthermore, since $M_t$ can be discontinuous only at $t = \tau_V$, the functional $M_t$ is quasi-left continuous in the sense of~\cite[III.3.14]{MR0264757}, and therefore $X^\psi_t$ is a Hunt process on $V$. However, we do not use these properties in our development.)

Informally, $X^\psi_t$ is obtained from $X_t$ by terminating the paths of $X_t$ with rate $\psi(X_t) dt$, and $\pi_\psi(x, dy)$ is the distribution of $X_t$ stopped at the time when $X^\psi_t$ is killed. Furthermore, $G_\psi(x, dy)$ is the potential kernel of $X^\psi_t$. To avoid technical difficulties related to subprocesses and the domains of their generators, in what follows we rely mostly on the formalism of additive and multiplicative functionals.

The multiplicative functional $\hat{M}_t$ is defined just as $M_t$, but for the dual process $\hat{X}_t$. We correspondingly define $\hat{\pi}_\psi$ and $\hat{G}_\psi$. Since the paths of $\hat{X}_t$ can be obtained from those of $X_t$ by time-reversal and $M_t$ and $\hat{M}_t$ are defined by integrals invariant 
upon time-reversal, the definition of $\hat{M}_t$ agrees with that of~\cite[formula~(13.24)]{MR2152573}. Hence, by~\cite[Theorem~13.25]{MR2152573}, $M_t$ and $\hat{M}_t$ are dual multiplicative functionals. It follows that the subprocess $\hat{X}^\psi_t$ of $\hat{X}_t$ corresponding to the multiplicative functional $\hat{M}_t$ is the dual process of $X^\psi_t$; see~\cite[13.6 and Remark~13.26]{MR2152573}. Hence, the potential kernel $G_\psi$ of $X^\psi_t$ admits a uniquely determined density function $G_\psi(x, y)$ ($x, y \in V$), which is excessive in $x$ with respect to the transition semigroup $T^\psi_t$ of $X^\psi_t$, and excessive in $y$ with respect to the transition semigroup $\hat{T}^\psi_t$ of $\hat{X}^\psi_t$. Furthermore, $\hat{G}_\psi(x, y) = G_\psi(y, x)$ is the density of the potential kernel of $\hat{X}^\psi_t$. Since $G_\psi(x, dy)$ is concentrated on $V$, we let $G_\psi(x, y) = 0$ if $x \in \X \setminus V$ or $y \in \X \setminus V$. Clearly, $G_\psi(x, dy)$ is dominated by $G_V(x, dy)$ for all $x \in V$, and therefore 
\[
 G_\psi(x, y) \le G_V(x, y),\quad x, y \in \X.
\]
There are important relations between $\pi_\psi$, $G_\psi$, $\psi$ and $\A$. If $f$ is nonnegative or bounded and \emph{vanishes} in $\X \setminus V$, then by Corollary~\ref{corr:gpsipsi} we have
\begin{align}
\label{eq:pig1}
 \pi_\psi f(x) & = G_\psi (\psi f)(x) , && x \in V .
\end{align}
Considering $\tau = \tau_V$, we note that $M(\tau_V) = 0$, and so for bounded or nonnegative $f$ 
\[
  \int_{[0, \tau_V]} f(X_t) d M_t = \int_{[0, \tau_V)} f(X_t) d M_t - f(X_{\tau_V}) M_{\tau_V-}.
\]
If $f \in \dom(\A)$, then formula~\eqref{eq:psid} gives
\begin{align}
\label{eq:pig2}
  G_\psi \A f(x) & = \pi_\psi f(x) - f(x) , && x \in V .
\end{align}
Furthermore, by~\eqref{eq:psigen}, for $f \in \dom(\A)$ we have
\begin{align*}
 G_\psi (\A - \psi) f(x) & = \ex_x (f(X_{\tau_V}) M_{\tau_V-}) - f(x) , && x \in V .
\end{align*}
In particular, if $f \in \dom(\A)$ vanishes outside of $V$, then we have
\begin{align}
\label{eq:pig3}
 G_\psi (\A - \psi) f(x) & = -f(x) , && x \in V
\end{align}
(which also follows directly from~\eqref{eq:pig1} and~\eqref{eq:pig2}). Formula~\eqref{eq:pig3} means that the generator of $X^\psi_t$ agrees with $\A - \psi$ on the intersection of the respective domains.

We now introduce the Green operators $G^\psi_U$ and harmonic measures $\pi^\psi_U$ for $X^\psi_t$. Let $U$ be an open subset of $V$. For nonnegative or bounded $f$ and $x \in V$ we let
\begin{align*}
 \pi^\psi_U f(x) & = \ex_x (f(X_{\tau_U}) M_{\tau_U-}) , & G^\psi_U f(x) & = \ex_x \int_0^{\tau_U} f(X_t) M_t dt .
\end{align*}
We note that $G^\psi_V f = G_\psi f$. Also, $\pi^\psi_V f = \pi_\psi f$, if $f$ vanishes in $V$. Furthermore, $G^\psi_U$ admits a density function $G^\psi_U(x, y)$, and we have $G^\psi_U(x, y) \le G_U(x, y)$, $G^\psi_U(x, y) \le G_\psi(x, y)$. If $f$ vanishes outside of $V$, then we can replace $M(\tau_U-)$ by $M(\tau_U)$ in the definition of $\pi^\psi_U$. By~\eqref{eq:psigen}, for any $f \in \dom(\A)$ we have
\begin{align}
\label{eq:psid2}
 \pi^\psi_U f(x) & = G^\psi_U (\A - \psi) f(x) + f(x) , && x \in V .
\end{align}
In particular, by an approximation argument,
\begin{align}
\label{eq:psiiw}
  \pi^\psi_U(x, E) & = \int_U G^\psi_U(x, y) \nu(y, E) m(dy) , && x \in U , \, E \sub \X \setminus \overline{U} .
\end{align}
Formulas~\eqref{eq:psid2} and~\eqref{eq:psiiw} can be viewed correspondingly as Dynkin's formula applied to the first exit time, and the Ikeda-Watanabe formula for $X^\psi_t$.

Recall that $x_0 \in \X$, $0 < r < p < q < R < R_0$, $B(x_0, q) \sub V \sub B(x_0, R)$, see Figure~\ref{fig:1}, $\ph \in \dom$ is positive in $V$ and vanishes in $\X \setminus V$, and $\ph(x) = 1$ for $x \in B(x_0, q)$.

\begin{lemma}
\label{lem:pipsiuest}
Let $U = V \setminus \overline{B}(x_0, q)$. If $(\A - \psi) \ph(x) \le 0$ for $x \in V$, then
\begin{align}
\label{eq:pipsiuest}
 \pi^\psi_U(x, V \setminus U) & \le \ph(x) , && x \in U .
\end{align}
\end{lemma}

\begin{proof}
By~\eqref{eq:psid2}, for $x \in U$ we have
\begin{align*}
 \pi^\psi_U \ph(x) - \ph(x) & = G^\psi_U (\A - \psi) \ph(x) \le 0.
\end{align*}
It remains to note that $\ph = 1$ on $V \setminus U$.
\end{proof}

Essentially, we use here (and later on) superharmonicity of $\ph$ with respect to $\A - \psi$.

\begin{lemma}
\label{lem:gpsi}
If $(\A - \psi) \ph(x) \le 0$ for $x \in V$, then
\begin{align}
\label{eq:gpsiest}
 G_\psi(x, y) & \le c_{\eqref{eq:gpsiest}} \ph(x) , && x \in V \setminus B(x_0, p) , \, y \in B(x_0, r) ,
\end{align}
where
\begin{align*}
 c_{\eqref{eq:gpsiest}}=c_{\eqref{eq:gpsiest}}(x_0, r, p, q, R) & = c_{\eqref{eq:green}}(x_0, r, p, R) + \frac{c_{\eqref{eq:tau}}(x_0, R) (c_{\eqref{eq:nu}}(x_0, p, q))^2}{m(B(x_0, p))} \, .
\end{align*}
\end{lemma}

\begin{proof}
Let $U = V \setminus \overline{B}(x_0, q)$ and $x \in U$. Let $f$ be a nonnegative function supported in $B(x_0, r)$, $\int f(y) m(dy) = 1$ and $g(z) = G_\psi f(z)$ (this is done to regularize $G_\psi(x, y)$). Using the definition of $G_\psi$, the relation $f(X_t) = 0$ for $t < \tau_U$ and the strong Markov property, we obtain that
\begin{align*}
 g(x) & = \ex_x \expr{\int_{\tau_U}^\infty f(X_t) M_t dt} = \ex_x (g(X_{\tau_U}) M_{\tau_U}) = \pi^\psi_U g(x) .
\end{align*}
We split the last expectation into two parts, corresponding to the events $X(\tau_U) \in B(x_0, p)$ and $X(\tau_U) \in \overline{A}(x_0, p, q)$ respectively. By~\eqref{eq:green} and the inequality $M(\tau_U) \le 1$, we have $g(z) \le c_{\eqref{eq:green}}(x_0, r, p, R)$ for $z \in \overline{A}(x_0, p, q)$. From~\eqref{eq:pipsiuest} it follows that
\begin{align}
\label{eq:gpsi2}
 \pi^\psi_U (g \ind_{\overline{A}(x_0, p, q)})(x) & \le c_{\eqref{eq:green}} \pi^\psi_U(x, \overline{B}(x_0, q)) \le c_{\eqref{eq:green}} \ph(x) .
\end{align}
For the other part, we use~\eqref{eq:psiiw} and~\eqref{eq:nu},
\begin{align*}
 \pi^\psi_U(g \ind_{B(x_0, p)})(x) & = \int_U \expr{\int_{B(x_0, p)} g(z) \nu(y, z) m(dz)} G^\psi_U(x, y) m(dy) \\
 & \le c_{\eqref{eq:nu}} \int_U \nu(y, x_0) G^\psi_U(x, y) m(dy) \cdot \int_{B(x_0, p)} g(z) m(dz) ,
\end{align*}
with constant $c_{\eqref{eq:nu}}(x_0, p, q)$. Using again~\eqref{eq:nu} and~\eqref{eq:psiiw}, and then~\eqref{eq:pipsiuest}, we obtain
\begin{align*}
 \int_U \nu(y, x_0) G^\psi_U(x, y) m(dy) & \le \frac{c_{\eqref{eq:nu}}}{m(B(x_0, p))} \int_U \nu(y, B(x_0, p)) G^\psi_U(x, y) m(dy) \\
 & = \frac{c_{\eqref{eq:nu}}}{m(B(x_0, p))} \, \pi^\psi_U(x, B(x_0, p)) \le \frac{c_{\eqref{eq:nu}} \ph(x)}{m(B(x_0, p))} \, .
\end{align*}
By~\eqref{eq:tau}, we have
\begin{align*}
 \int_{B(x_0, p)} g(z) m(dz) & \le \int_{B(x_0, r)} \expr{\int_{B(x_0, p)} G_V(z, y) m(dz)} f(y) m(dy) \\
 & \le \int_{B(x_0, r)} \hat{\ex}_y(\hat{\tau}_V) f(y) m(dy) \le c_{\eqref{eq:tau}} ,
\end{align*}
with constant $c_{\eqref{eq:tau}}(x_0, R)$. Hence,
\begin{align*}
 \pi^\psi_U(g \ind_{B(x_0, p)})(x) & \le \frac{(c_{\eqref{eq:nu}})^2 c_{\eqref{eq:tau}} \ph(x)}{m(B(x_0, p))} \, .
\end{align*}
This and~\eqref{eq:gpsi2} yield that $g(x) \le c_{\eqref{eq:gpsiest}} \ph(x)$, with $c_{\eqref{eq:gpsiest}}$ given in the statement of the lemma.

Recall that $g = G_\psi f$, where $f$ is an arbitrary nonnegative function vanishing outside $B(x_0, r)$ with integral equal to $1$. Hence, by approximation, for each $x \in \X \setminus \overline{B}(x_0, q)$, formula~\eqref{eq:gpsiest} holds for \emph{almost every} $y \in B(x_0, r)$. By Proposition~\ref{prop:aa} (applied to $X^\psi_t$), \eqref{eq:gpsiest} holds for \emph{every} $y \in B(x_0, r)$.

For $x \in \overline{A}(x_0, p, q)$, the result follows easily from~\eqref{eq:green}. Indeed, we have $G_\psi(x, y) \le G_V(x, y) \le c_{\eqref{eq:green}} = c_{\eqref{eq:green}} \ph(x)$, with constant $c_{\eqref{eq:green}}(x_0, r, p, R)$. Hence, formula~\eqref{eq:gpsiest} holds also for $x \in \overline{A}(x_0, p, q)$, with the same constant.
\end{proof}

The above arguments can be repeated for the dual process $\hat{X}_t$. Hence, the \emph{dual} versions of Lemmas~\ref{lem:pipsiuest} and~\ref{lem:gpsi} hold true, with the same $c_{\eqref{eq:gpsiest}}$.

We are very close to the estimate of $\pi_\psi(x, dy)$ for $x \in B(x_0, r)$. Indeed, for $y \in V$ we have $\pi_\psi(x, dy) = G_\psi(x, y) \psi(y) m(dy)$ (see~\eqref{eq:pig1}). When $y \in \X \setminus \overline{V}$, then, at least heuristically, $\pi_\psi(x, dy) = \hat{\A} G_\psi^x(y) m(dy)$, where $G_\psi^x(y) = G_\psi(x, y)$ vanishes outside of $V$ (see~\eqref{eq:pig2}). This will give satisfactory bounds when $y \in \X \setminus V$. Before we proceed, we first show that $\pi_\psi(x, \partial V) = 0$.

\begin{lemma}
\label{lem:pipsiint}
Suppose that for some $c_3, c_4 > 0$, we have $\psi(x) \ge c_3 + (\ph(x))^{-1} \hat{\A} \ph(x)$ and $\psi(x) \le c_4 / \ph(x)$ for $x \in V$. Then for every nonnegative function $f$ we have
\begin{align}
\label{eq:pipsiint}
 c_3 \int_V \pi_\psi f(x) \ph(x) m(dx) & \le c_4 \int_V f(x) m(dx) + \int_{\X \setminus V} f(x) \hat{\A} \ph(x) m(dx) .
\end{align}
\end{lemma}

\begin{proof}
First, suppose that $f \in \dom(\A)$. Denote $h(x) = -(\hat{\A} - \psi) \ph(x)$ for $x \in V$. Note that $h$ is nonnegative. Let $g(x) = \pi_\psi f(x)$ for $x \in \X$; hence $g(x) = f(x)$ for $x \in \X \setminus V$, see~\eqref{eq:pipsi2}. By~\eqref{eq:pig2}, we have $g(x) = f(x) + G_\psi \A f(x)$ for $x \in V$. Hence,
\begin{align*}
 \int_V g(x) h(x) m(dx) & = \int_V f(x) h(x) m(dx) + \int_V G_\psi \A f(x) h(x) m(dx) .
\end{align*}
For the second term, we have
\begin{align*}
 \int_V G_\psi \A f(x) h(x) m(dx) & = \int_V \A f(x) \hat{G}_\psi h(x) m(dx) .
\end{align*}
By~\eqref{eq:pig3} (dual version), $\hat{G}_\psi h(x) = -\hat{G}_\psi (\hat{\A} - \psi) \ph(x) = \ph(x)$ for $x \in V$. Hence,
\begin{align*}
 \int_V G_\psi \A f(x) h(x) m(dx) & = \int_V \A f(x) \ph(x) m(dx) = \int_\X f(x) \hat{\A} \ph(x) m(dx) .
\end{align*}
In the last equality, we used the fact that $\ph(x) = 0$ for $x \in \X \setminus V$. It follows that
\begin{align*}
 \int_V g(x) h(x) m(dx) & = \int_V f(x) h(x) m(dx) + \int_\X f(x) \hat{\A} \ph(x) m(dx) .
\end{align*}
But $h(x) = -(\hat{\A} - \psi) \ph(x)$, so that finally, after simplification,
\begin{align*}
 \int_V g(x) h(x) m(dx) & = \int_V f(x) \ph(x) \psi(x) m(dx) + \int_{\X \setminus V} f(x) \hat{\A} \ph(x) m(dx) .
\end{align*}
Using the inequalities $\psi(x) \ph(x) \le c_4$ for $x \in V$ and $h(x) = \psi(x) \ph(x) - \hat{\A} \ph(x) \ge c_3 \ph(x)$ for $x \in V$, we obtain~\eqref{eq:pipsiint}. The general case of nonnegative $f$ (not necessarily in $\dom(\A)$) follows by approximation.
\end{proof}

\begin{lemma}
\label{lem:boundary2}
Suppose that for some $c_3, c_4 > 0$, we have $\psi(x) \ge c_3 + (\ph(x))^{-1} \hat{\A} \ph(x)$ and $\psi(x) \le c_4 / \ph(x)$ for $x \in V$. Then $M(\tau_V-) \ind_{\partial V}(X(\tau_V)) = 0$ $\pr_x\as$ and $\pi_\psi(x, \partial V) = 0$ for all $x \in V$.
\end{lemma}

\begin{proof}
For $x \in V$ define $g(x) = \pi_\psi(x, \partial V)$. By~\eqref{eq:pipsiint}, $\int_V g(x) \ph(x) m(dx) = 0$, so that $g$ vanishes almost everywhere in $V$. We claim that $g$ is excessive for the transition semigroup $T^\psi_t$ of $X^\psi_t$. Indeed, we have $g(x) = \ex_x(M(\tau_V-); X(\tau_V) \in \partial V)$, so that by the Markov property, for any $t > 0$ and $x \in V$,
\begin{align*}
 \ex_x (M_t g(X_t)) & = \ex_x (M_t g(X_t) ; t < \tau_V) = \ex_x(M_{\tau_V-}; X_{\tau_V} \in \partial V , \, t < \tau_V) .
\end{align*}
The right-hand side does not exceed $g(x)$, and by monotone convergence, it converges to $g(x)$ as $t \searrow 0$. Hence $g$ is an excessive function equal to zero almost everywhere in $V$. By~\cite{MR0264757}, Proposition~II.3.2 (or by Proposition~\ref{prop:aa}), $g(x) = 0$ for all $x \in V$.
\end{proof}

Recall that according to the remark following Lemma~\ref{lem:boundary1}, we keep assuming that $\psi(x) \ge c_1 (\ph(x))^{-1} - c_2$ for $x \in V$. Consider $\tilde{\psi}(x) = c_1^{-1} \delta (\psi(x) + c_2) + c_3$ for some $c_3 > 0$, and let $\tilde{M}_t$ be the multiplicative functional defined in a similar manner as $M_t$, but with $\psi$ replaced by $\tilde{\psi}$. Clearly, for all $t > 0$ we have $M_t = 0$ if and only if $\tilde{M}_t = 0$. Since $\tilde{\psi}(x) \ge c_3 + \delta / \ph(x)$, an application of Lemma~\ref{lem:boundary2} to $\tilde{\psi}$ yields the following result.

\begin{corollary}
\label{corr:boundary2}
Suppose that for some $c > 0$, we have $\psi(x) \le c / \ph(x)$ for $x \in V$. Then $M(\tau_V-) \ind_{\partial V}(X(\tau_V)) = 0$ $\pr_x\as$ for $x \in V$. In particular, $\pi_\psi(x, \partial V) = 0$ for $x \in V$.
\qed
\end{corollary}

Now we make the actual choice of $\psi$.

\begin{lemma}
\label{lem:pipsiest}
Let $\delta$ be given by \eqref{eq:ddelta}, and
\begin{align}
\label{eq:psi}
 \psi(x) & = \frac{\max(\A \ph(x), \hat{\A} \ph(x), \delta (1 - \ph(x)))}{\ph(x)} \, , && x \in \X \cup \set{\partial} ,
\end{align}
where $1 / 0 = \infty$. For all $x \in B(x_0, r)$ we have $\pi_\psi(x, dy) \le \tilde\pi_\psi(y) m(dy)$, where
\begin{align}
\label{eq:pipsiest}
 \tilde\pi_\psi(y) = c_{\eqref{eq:gpsiest}} \expr{\delta \ind_{V \setminus B(x_0, q)}(y) + 2 \min(\delta, \hat{\nu}(y, V)) \ind_{\X \setminus V}(y)}
\end{align}
with $c_{\eqref{eq:gpsiest}} = c_{\eqref{eq:gpsiest}}(x_0, r, p, q, R)$ given in Lemma~\ref{lem:gpsi}.
\end{lemma}

\begin{proof}
Note that $\psi(x) \ge \delta (\ph(x))^{-1} - \delta$, $(\A - \psi) \ph(x) \le 0$, $(\hat{\A} - \psi) \ph(x) \le 0$ and $\psi(x) \le \delta / \ph(x)$ for $x \in V$. Hence, we may apply Lemmas~\ref{lem:boundary1}, \ref{lem:pipsiuest} and~\ref{lem:gpsi}, Corollary~\ref{corr:boundary2}, and their dual versions. By Corollary~\ref{corr:boundary2}, $\pi_\psi(x, \partial V) = 0$ for all $x \in V$. Since $\A \ph(x) \le 0$ and $\hat{\A} \ph(x) \le 0$ for $x \in B(x_0, q)$, we have $\psi(x) = 0$ for $x \in B(x_0, q)$, and therefore $\pi_\psi(x, B(x_0, q)) = 0$ for all $x \in V$.

Fix $x \in B(x_0, r)$. If $f$ is nonnegative and vanishes in $B(x_0, q)$ and in $\X \setminus V$, then~\eqref{eq:pig1} yields that
\begin{align*}
 \pi_\psi f(x) & = G_\psi (\psi f)(x) = \int_{V \setminus B(x_0, q)} G_\psi(x, y) \psi(y) f(y) m(dy) .
\end{align*}
Using~\eqref{eq:gpsiest} for $\hat{G}_\psi$ and the inequality $\ph(y) \psi(y) \le \delta$ for $y \in V$, we have
\begin{align}
\label{eq:pipsiest2}
 \pi_\psi f(x) & \le c_{\eqref{eq:gpsiest}} \int_{V \setminus B(x_0, q)} \ph(y) \psi(y) f(y) m(dy) \le c_{\eqref{eq:gpsiest}} \delta \int_{V \setminus B(x_0, q)} f(y) m(dy) ,
\end{align}
with constant $c_{\eqref{eq:gpsiest}}(x_0, r, p, q, R)$. Suppose now that $f \in \dom(\A)$ vanishes in $V$. By~\eqref{eq:pig2},
\begin{align*}
 \pi_\psi f(x) & = G_\psi \A f(x) = \int_V G_\psi(x, y) \expr{\int_{\X \setminus V} f(z) \nu(y, z) m(dz)} m(dy) \\
 & = \int_{\X \setminus V} \expr{\int_V G_\psi(x, y) \nu(y, z) m(dy)} f(z) m(dz) .
\end{align*}
We estimate the inner integral for $z \in \X \setminus V$. Using~\eqref{eq:gpsiest} for $\hat{G}_\psi$, we have
\begin{align*}
 \int_{V \setminus B(x_0, p)} G_\psi(x, y) \nu(y, z) m(dy) & \le c_{\eqref{eq:gpsiest}} \int_{V \setminus B(x_0, p)} \ph(y) \nu(y, z) m(dy) = c_{\eqref{eq:gpsiest}} \hat{\A} \ph(z) .
\end{align*}
The integral over $B(x_0, p)$ is estimated as in the proof of Lemma~\ref{lem:gpsi},
\begin{align*}
 \int_{B(x_0, p)} G_\psi(x, y) \nu(y, z) m(dy) & \le c_{\eqref{eq:nu}} \nu(x_0, z) \int_{B(x_0, p)} G_V(x, y) m(dy) \\
 & \le c_{\eqref{eq:nu}} \nu(x_0, z) \ex_x \tau_V \le c_{\eqref{eq:nu}} c_{\eqref{eq:tau}} \nu(x_0, z) \\
 & \le \frac{c_{\eqref{eq:tau}} (c_{\eqref{eq:nu}})^2}{m(B(x_0, p))} \, \hat{\nu}(z, B(x_0, p)) \le c_{\eqref{eq:gpsiest}} \hat{\A} \ph(z) ,
\end{align*}
with constants $c_{\eqref{eq:nu}}(x_0, p, q)$, $c_{\eqref{eq:tau}}(x_0, R)$ and $c_{\eqref{eq:gpsiest}}(x_0, r, p, q, R)$. Since $\hat{\A} \ph(z) \le \delta$ and $\hat{\A} \ph(z) \le \hat{\nu}(z, V)$, we obtain that
\begin{align}
\label{eq:pipsiest3}
 \pi_\psi f(x) & {\le 2 c_{\eqref{eq:gpsiest}} \int_{\X \setminus V} f(z) \hat{\A} \ph(z) m(dz)} \le {2} c_{\eqref{eq:gpsiest}} \int_{\X \setminus V} f(z) \min(\delta, \hat{\nu}(z, V)) m(dz) .
\end{align}
By approximation, \eqref{eq:pipsiest3} holds for any nonnegative $f$ vanishing in $\overline{V}$. Formula~\eqref{eq:pipsiest} is a combination of~\eqref{eq:pipsiest2}, \eqref{eq:pipsiest3}, $\pi_\psi(x, \partial V) = 0$ and $\pi_\psi(x, B(x_0, q)) = 0$ for all $x \in V$.
\end{proof}

\begin{lemma}
\label{lem:subharmonic}
If a nonnegative function $f$ is regular subharmonic in $B(x_0, R)$, then $f(x) \le \pi_\psi f(x)$ for $x \in B(x_0, r)$. If $f$ is regular harmonic, then equality holds.
\end{lemma}

\begin{proof}
If $f$ is regular subharmonic in $V$, then 
$f(x) \le \ex_x (f(X(\tau_a)))$ for all $a \in [0, \infty]$. If $f$ is regular harmonic in $V$, then equality holds. The result follows by~\eqref{eq:pipsi1}.
\end{proof}

The local maximum estimate is now proved as follows.

\begin{proof}[Proof of Theorem~\ref{th:regexit}]
Fix $p \in (r, q)$. Choose $\eps > 0$ and $\ph$ as in the beginning of this section, and so that $\delta = \sup_{x \in \X} \max(\A \ph(x), \hat{\A} \ph(x)) < \ro(\overline{B}(x_0, q), B(x_0, R)) + \eps$. Define $\psi$ as in~\eqref{eq:psi}. By Lemmas~\ref{lem:pipsiest} and~\ref{lem:subharmonic}, we have~\eqref{eq:regexit1} with $\pi_{x_0, r, q, R}(y)$ bounded from above by $\pi_\psi(y)$ defined in~\eqref{eq:pipsiest}. Note that $\hat{\nu}(y, V) \le \hat{\nu}(y, B(x_0, R))$. Since $\eps > 0$ and $p \in (r, q)$ are arbitrary, formulas~\eqref{eq:regexit2} and~\eqref{eq:regexit3} follow.
\end{proof}

We conclude this section with a result on diffusion processes. The above argument remains valid when $\nu$ vanishes everywhere, i.e., $X_t$ is a diffusion process. In this case~\eqref{eq:tau} is not a consequence of Assumption~\ref{assu:levy}, so we need to add~\eqref{eq:tau} as an assumption. No other changes in the argument are needed, and in fact the proof of Lemma~\ref{lem:gpsi} simplifies significantly, since $X_t$ exits $U$ through the boundary of $U$, and therefore $X(\tau_U)$ is never in $B(x_0, p)$. Therefore, we have proved the following result.

\begin{theorem}
\label{th:diffusions}
Assume that $X_t$ is a diffusion process satisfying Assumptions~\ref{assu:dual}, \ref{assu:gen} and~\ref{assu:green}, and formula~\eqref{eq:tau}. Let $x_0 \in \X$ and $0 < r < q < R < R_0$, where $R_0$ is the localization radius of~\ref{eq:tau} and Assumption~\ref{assu:green}. Let $f$ be a nonnegative function on $B(x_0, R)$, regular subharmonic in $B(x_0, R)$ with respect to $X_t$. Then
\begin{align}
\label{eq:diffusions}
 f(x) & \le c_{\eqref{eq:diffusions}} \int_{\overline{A}(x_0, q, R)} f(y) m(dy) , && x \in B(x_0, r) .
\end{align}
Here $c_{\eqref{eq:diffusions}} = c_{\eqref{eq:diffusions}}(x_0, r, q, R) = c_{\eqref{eq:green}} \delta$, where $\delta = \ro(\overline{B}(x_0, q), B(x_0, R))$ and $c_{\eqref{eq:green}} = c_{\eqref{eq:green}}(x_0, r, q, R)$ are defined in Assumptions~\ref{assu:gen} and~\ref{assu:green}.
\end{theorem}

\begin{remark}
\label{rem:dirichlet}
For diffusion processes, local supremum estimate~\eqref{eq:diffusions} for subharmonic functions is typically proved analytically, using Sobolev embeddings and Moser iteration, see, e.g., \cite{MR737190}. Theorem~\ref{th:regexit} requires more regularity of the process $X_t$ as compared to the analytical approach because we assume the existence of bump functions in the domain of the Feller generator (Assumption~\ref{assu:gen}), while Moser iteration is based on the energy form. However, our approach does not depend on Sobolev embeddings, and so it applies also to Sierpi{\'n}ski carpets and some other highly irregular state spaces $\X$. It would be interesting to find an analytical proof of the local supremum estimate for jump-type processes, which would not require Assumption~\ref{assu:gen}. Related results have been recently studied when the L\'evy kernel $\nu(x, y)$ is comparable to $(d(x, y))^{-d - \alpha}$ (see~\cite{MR2448308} and the references therein). Further comments on this subject are given in Example~\ref{ex:stable-like} and Appendix~\ref{sec:app}.
\end{remark}

%
%

\section{Extensions and examples}
\label{sec:ex}

In this section we study several applications of our boundary Harnack inequality, and discuss limitations of Theorem~\ref{th:bhi}. We sketch the range of possible applications by indicating rather general classes of processes satisfying the assumptions of Theorem~\ref{th:bhi}, without getting into technical details. Before that, however, we discuss an important notion of scale-invariance introduced in Remark~\ref{rem:si}. This property can be proved in a fairly general setting, which we call \emph{stable-like scaling}.

\begin{definition}
\label{def:stable-like}
The process $X_t$ is said to have \emph{stable-like scaling} property with dimension $n > 0$, index $\alpha > 0$ and localization radius $R_0 \in (0, \infty]$ ($\alpha$-stable-like scaling in short), if the following conditions are met:
\begin{enumerate}[label={\rm (\alph{*})},ref={\rm (\alph{*})}]
\item \label{it:reg}
$\X$ is locally an Ahlfors regular $n$-space; that is, $c^{-1} r^n \le m(B(x, r)) \le c r^n$ when $0 < r < R_0$ and $x \in \X$;
\item \label{it:nu}
$c_{\eqref{eq:nu}}(x_0, r, R) \le c(r/R)$ when $0 < r < R < R_0$, $x_0 \in \X$ in the relative constancy of the L\'evy measure condition in Assumption~\ref{assu:levy};
\item \label{it:nu2}
$c_{\eqref{eq:nu2}}(x_0, r, R) \ge c(r/R) R^{-n - \alpha}$ when $0 < r < R < R_0$, $x_0 \in \X$, that is, $\nu(x, y) \ge c (d(x, y))^{-n - \alpha}$ when $d(x, y) < R_0$;
\item \label{it:tau}
$c_{\eqref{eq:tau}}(x_0, r) \le c r^\alpha$ when $0 < r < R_0$, $x_0 \in \X$ in the upper bound for mean exit time from a ball;
\item \label{it:green}
$c_{\eqref{eq:green}}(x_0, r, p, R) \le c(r/R, p/R) R^{\alpha - n}$ when $0 < r < p < R < R_0$ and $x_0 \in \X$ in the off-diagonal upper bound for the Green function of a ball;
\item \label{it:gen}
$\ro(\overline{B}(x_0, r), B(x_0, R)) \le c(r/R) R^{-\alpha}$ when $0 < r < R < R_0$ and $x_0 \in \X$, and $\ro(\overline{A}(x_0, p, R), A(x_0, r, \tilde{r})) \le c(r/R, p/R, R/\tilde{r}) R^{-\alpha}$ when $0 < r < p < R < \tilde{r}$ in Assumption~\ref{assu:gen}.
\end{enumerate}
\end{definition}

\begin{proposition}
\label{prop:stable-like}
If the scaling property~\itref{it:reg} is satisfied, then conditions~\itref{it:nu}, \itref{it:nu2} and~\itref{it:tau} are consequences of:
\begin{enumerate}[label={\rm (\alph{*})},ref={\rm (\alph{*})}]
\setcounter{enumi}{6}
\item \label{it:nu3}
the L\'evy kernel of $X_t$ satisfies
\[
 c^{-1} (d(x, y))^{-n - \alpha} \exp(-q d(x, y)) \le \nu(x, y) \le c (d(x, y))^{-n - \alpha} \exp(-q d(x, y))
\]
for some $q \ge 0$ and for all $x, y \in \X$.
\end{enumerate}
\end{proposition}

Note that the same parameter $q$ appears in the lower and the upper bound.

\begin{proof}
Conditions~\itref{it:nu} and~\itref{it:nu2} follow directly from~\itref{it:nu3}. Furthermore, by~\itref{it:reg} and the triangle inequality, there is $R_0 > 0$ such that if $x_0 \in \X$ and $0 < r < R_0$, then for some $y \in B(x_0, c_1 r)\setminus B(x_0, r)$ where $c_1 > 2$, the balls $B(x_0, r)$ and $B(y, r)$ are disjoint. Hence, for all $x \in B(x_0, r)$ we have by \itref{it:reg} and \itref{it:nu3}, $\nu(x, \X \setminus B(x_0, r)) \ge \nu(x, B(y, r)) \ge c_2 r^{-\alpha}$. As in the proof of Proposition~\ref{prop:tau}, it follows that $\mbox{$\pr_x(\tau_{B(x_0, r)} > t)$} \le \exp(-c_2 r^{-\alpha} t)$, and therefore $\ex_x(\tau_{B(x_0, r)}) \le c_2^{-1} r^\alpha$.
\end{proof}

We also have the following sufficient condition for scaling properties~\itref{it:tau} and~\itref{it:green}.

\begin{proposition}
\label{prop:green}
Assume that scaling property~\itref{it:reg} holds. Suppose that the transition density $T_t(x, y)$ of a Hunt process $X_t$ exists, and that for some $\alpha > 0$, $r_0 > 0$,
\begin{align}
\label{eq:ttest}
 \frac{1}{c_{\eqref{eq:ttest}}} \, \min \expr{t^{-n/\alpha}, \frac{t}{(d(x, y))^{n + \alpha}}} \le T_t(x, y) & \le c_{\eqref{eq:ttest}} \min \expr{t^{-n/\alpha}, \frac{t}{(d(x, y))^{n + \alpha}}}
\end{align}
for $x, y \in \X$ with $d(x, y) < r_0$, and any $t \in (0, r_0^\alpha)$. Then Assumption~\ref{assu:green} and scaling conditions~\itref{it:tau} and~\itref{it:green} hold. The constant $c_{\eqref{eq:green}}$ and the localization radius $R_0$ in~\eqref{eq:green} depend only on the constants in~\eqref{eq:ttest} (including $\alpha$ and $r_0$) and in the Ahlfors regularity condition.
\end{proposition}

\begin{proof}
Both cases $\alpha > n$ and $\alpha < n$ are very similar (in fact, slightly simpler) to the remaining case $\alpha = n$. Hence we give a detailed argument only when $\alpha = n$.

With no loss of generality we may assume that $r_0 < \diam \X$. We choose $k > 2$ so that $m(B(x_0, k r) \setminus B(x_0, r)) \ge r^n$ for all $x_0 \in \X$ and $r < r_0 / k$. Let $r < r_0 / (1 + k)^{1 + 1/\alpha}$, $x_0 \in \X$, $D = B(x_0, r)$, and let $T_t^D$ be the transition kernel of the killed process $X^D_t$. Recall that $G_D(x, y) = \int_0^\infty T^D_t(x, y) dt$. Let $x, y \in D$ and let $t_1 = (d(x, y))^\alpha$, $t_2 = (2 r)^\alpha$. Since $d(x, y) < 2 r < r_0$, we have
\[
 \int_0^{t_1} T_t^D(x, y) dt \le \frac{c_{\eqref{eq:ttest}}}{(d(x, y))^{n + \alpha}} \int_0^{t_1} t dt = \frac{c_{\eqref{eq:ttest}} (d(x, y))^{\alpha - n}}{2} \, ,
\]
and
\[
 \int_{t_1}^{t_2} T_t^D(x, y) dt \le c_{\eqref{eq:ttest}} \int_{t_1}^{t_2} t^{-n/\alpha} dt = \alpha c_{\eqref{eq:ttest}} \log \frac{2 r}{d(x, y)} \, .
\]
Note that for $\alpha > n$ or $\alpha < n$, we simply have a different expression for the above integral. When $t \in [t_2, 2 t_2]$, we have $t < 2 t_2 < 2^{1 + \alpha} r^\alpha < r_0^\alpha$, and hence $T_t^D(x, y) \le c_{\eqref{eq:ttest}} {t_2}^{-n / \alpha} = c_{\eqref{eq:ttest}} \, (2 r)^{-n}$. Furthermore, since  $T_{t_2}\ind\le \ind$ and $d(x, z) < (1 + k) r < r_0$ for $z \in B(x_0, k r)$,
\begin{align*}
 T_{t_2}^D \ind(x) & \le T_{t_2} \ind_D(x) \le 1 - \int_{\X \setminus D} T_{t_2}(x, z) m(dz) \\
 & \le 1 - \frac{1}{c_{\eqref{eq:ttest}}} \int_{B(x_0, k r) \setminus B(x_0, r)} \frac{t_2}{(d(x, z))^{n + \alpha}} \, m(dz) \\
 & \le 1 - \frac{2^\alpha m(B(x_0, k r) \setminus B(x_0, r))}{c_{\eqref{eq:ttest}} \, (k + 1)^{n+\alpha} r^n} \le 1 - \frac{2^\alpha}{c_{\eqref{eq:ttest}} \, (k + 1)^{n+\alpha}} \, .
\end{align*}
For $s = j t_2 + t$, $t \in [t_2, 2 t_2]$, $j \ge 0$, we have $T_s^D = (T_{t_2}^D)^j T_t^D$. It follows that
\[
 T_s^D(x, y) \le \expr{1 - \frac{2^\alpha}{c_{\eqref{eq:ttest}} \, (k + 1)^{n+\alpha}}}^j \frac{c_{\eqref{eq:ttest}}}{(2 r)^n} \, ,
\]
and therefore, by summing up a geometric series,
\[
 \int_{t_2}^\infty T_t^D(x, y) dt \le 2^{-\alpha}(c_{\eqref{eq:ttest}})^2 (k + 1)^{n+\alpha} t_2 (2 r)^{-n} = 2^{-n} (k + 1)^{n+\alpha} (c_{\eqref{eq:ttest}})^2 r^{\alpha - n}.
\]
We conclude that $G_D(x, y) \le (c_{\eqref{eq:ttest}} / 2) + n c_{\eqref{eq:ttest}}\log (2r / d(x, y)) + 2^{-n} (k + 1)^{2n} (c_{\eqref{eq:ttest}})^2$. This gives Assumption~\ref{assu:green} and property~\itref{it:green}. Property~\itref{it:tau} follows by simple integration.
\end{proof}

If $X_t$ has $\alpha$-stable-like scaling, then, by a simple substitution, in Theorem~\ref{th:bhi} we have
\begin{align*}
 & c_{\eqref{eq:regexit3}}(x_0, r, q, R) \le c(r/R, q/R) R^{\alpha - d} , \\
 & c_{\eqref{eq:ubhic2}}(x_0, q, R) \le c(q/R) R^\alpha , \\
 & c_{\eqref{eq:ubhi}}(x_0, r, R) \le c(r/R) .
\end{align*}
Hence the boundary Harnack inequality is {uniform} in all scales $R \in (0,R_0)$, or \emph{scale-invariant}, as claimed in Remark~\ref{rem:si}. We state this result as a separate theorem for future reference.

\begin{theorem}
\label{th:sibhi}
If the assumptions of Theorem~\ref{th:bhi} are satisfied, and the process $X_t$ has $\alpha$-stable-like scaling, then the boundary Harnack inequality~\itref{it:bhi} is scale-invariant: $c_{\eqref{eq:ubhi}}(x_0, r, R)$ depends only on $r / R$.
\end{theorem}

In typical applications, one verifies (typically quite straightforward) conditions~\itref{it:reg} and~\itref{it:nu3}, formula~\eqref{eq:ttest} (which has been proved for a fairly general class of processes), and condition~\itref{it:gen}. When dealing with processes given the L\'evy kernel $\nu(x, y)$, condition~\itref{it:gen} turns out to be the most restrictive one.

\begin{example}[L\'evy processes]
\label{ex:levy}
Theorem~\ref{th:bhi} applies to a large class of L\'evy processes. In this case, the notion of processes in duality and properties of the Feller generator simplify significantly, see~\cite{MR1739520}.

Let $X_t$ be a L\'evy process in $\X = \R^k$ (with the Euclidean distance $d$ and Lebesgue measure $m$). Then $X_t$ is always Feller, and it is strong Feller if and only if the distribution of $X_t$ is absolutely continuous (with respect to the Lebesgue measure). If this is the case, Assumption~\ref{assu:dual} is satisfied: the dual of $X_t$ exists, and it is the reflected process, $\hat{X}_t - \hat{X}_0 = -(X_t - X_0)$. Assumption~\ref{assu:gen} is always satisfied with $\dom = C_c^\infty(\R^k)$. The L\'evy kernel of $X_t$ is translation-invariant, $\nu(x, E) = \nu(E - x)$, where $\nu(dz)$ is the L\'evy measure of $X_t$. Therefore, Assumption~\ref{assu:levy} can be restated as follows: the L\'evy measure of $X_t$ is absolutely continuous, and its density function $\nu(z)$ satisfies
\begin{align}
\label{eq:nu:levy}
 0<c_{\eqref{eq:nu}}^{-1} \, \nu(z_0) \le \nu(z) & \le c_{\eqref{eq:nu}} \, \nu(z_0) , && |z_0| > R , \, |z - z_0| < r ,
\end{align}
whenever $0 < r < R$, with constant $c_{\eqref{eq:nu}}(0, r, R)$. 
If, e.g., $\nu(z)$ is isotropic and radially non-increasing, then~\eqref{eq:nu:levy} is equivalent to $\nu(z_2) \ge c \nu(z_1)>0$ being valid whenever {$z_1,z_2\in \R^k$,} $|z_1| \ge 1$ and $|z_2|=|z_1|+1$.
Indeed, let us assume the latter condition. By radial monotonicity, $\nu$ is locally bounded on $\R^k \setminus \{0\}$ from above and below by positive constants. Therefore, $c_1 = c_1(c,\nu,r, R) > 0$ exists such that 
$\nu(z_2) \ge c_1 \nu(z_1)$ if $|z_1| \ge R-r$ and $|z_2|=|z_1|+1$. If follows that 
$(c_1)^n \nu(z_1) \le \nu(z_2) \le \nu(z_1)$ if $R-r \le |z_1| \le |z_2| \le |z_1| + n$ and $n = 1, 2, \ldots$ Taking $n\ge r$ we obtain \eqref{eq:nu:levy}, as desired.
Finally, Assumption~\ref{assu:green} in many cases follows from estimates of the potential kernel $\pot(x, y) = \pot(y - x)$, or, in the recurrent case, the $\alpha$-potential kernel $\pot_\alpha(x, y) = \pot_\alpha(y - x)$, see 
Proposition~\ref{prop:green:pot}.

We conclude that boundary Harnack inequality holds for a L\'evy process $X_t$, provided that its L\'evy measure satisfies~\eqref{eq:nu:levy}, one-dimensional distributions of $X_t$ are absolutely continuous, and the Green functions of balls satisfy Assumption~\ref{assu:green}. This class includes:
\begin{itemize}
\item subordinate Brownian motions which are not compound Poisson processes and have non-zero L\'evy measure density function satisfying $\nu(z_2) \ge c \nu(z_1)$ if $|z_1| \ge 1$ and $|z_2|=|z_1|+1$ (for properties of these processes, see, e.g., \cite{MR2569321, Kim12potentialtheory});
\item (possibly asymmetric) L\'evy processes with non-degenerate Brownian part and L\'evy measure satisfying~\eqref{eq:nu:levy};
\item (possibly asymmetric) strictly stable L\'evy processes, whose L\'evy measure is of the form $|z|^{-d - \alpha} f(z / |z|) dz$ for a function $f$ bounded from below and above by positive constants.
\end{itemize}
Scale-invariance depends on more accurate estimates. We give here some examples and directions.
\begin{itemize}
\item For the class of strictly stable L\'evy processes just mentioned above, scale-invariance follows from the estimates of the transition density given in~\cite[Theorem~1.1]{MR2286060} and Proposition~\ref{prop:green}; see also~\cite{MR2320691} and the references therein for related estimates in the symmetric (but anisotropic) case.
\item Some L\'evy processes for which Theorem~\ref{th:bhi} gives scale-invariant BHI are included in Example~\ref{ex:stable-like} (stable-like L\'evy processes) and Example~\ref{ex:perturbations} (mixtures of isotropic stable processes, relativistic stable processes, etc.).
\item A non-scale-invariant case (mixture of an isotropic stable process and the Brownian motion) is discussed in Example~\ref{ex:brownian}.
\item Our results may be used to recover the recent scale-invariant BHI given in \cite[Theorem~1.1]{MR2994122}.  More specifically, using our results and the first part of~\cite{MR2994122}, one can obtain scale-invariant BHI for the  L\'evy processes considered therein (isotropic L\'evy processes with 
the L\'evy measure comparable to that of a rather general subordinate Brownian motion with some scaling properties), thus replacing the second part of \cite{MR2994122} and significantly simplifying the whole argument of \cite{MR2994122}.
\item Similarly, the estimates given in~\cite{2012arXiv1202.0706K}, combined with Theorem~\ref{th:bhi},  should give a compact proof of scale-invariant BHI for a class of subordinate Brownian motions with L\'evy-Khintchine exponent slowly varying at~$\infty$. 
\end{itemize}
\end{example}

\begin{example}[Stable-like processes]
\label{ex:stable-like}
Let $\X$ be a closed set in $\R^k$, and let $m$ be a measure on $\X$ such that $\X$, with the Euclidean distance, is an Ahlfors regular $n$-space for some $n > 0$. For example, $\X$ can be entire $\R^k$ or the closure of an open set in $\R^k$ (with the Lebesgue measure $m$; then $n = k$). On the other hand, $\X$ can be a fractal set, such as Sierpi\'nski gaskets ($n = \log (k + 1) / \log 2$) or Sierpi\'nski carpets ($n = \log (3^k - 1) / \log 3$) in $\R^2$, equipped with an appropriate Hausdorff measure. By this assumption, scaling property~\itref{it:reg} is satisfied.

Let $\alpha \in (0, 2)$, and suppose that $\nu(x, y) = \nu(y, x)$ and
\begin{align}
\label{eq:stablelevy}
 c_1 |x - y|^{-n - \alpha} \le \nu(x, y) & \le c_2 |x - y|^{-n - \alpha} , && x, y \in \X .
\end{align}
This immediately gives Assumption~\ref{assu:levy} with scaling property~\itref{it:nu3}.

By~\cite[Theorem~1]{MR2008600}, there is a Feller, strong Feller, symmetric pure-jump Hunt process $X_t$ with L{\'e}vy kernel $\nu$, and the continuous transition probability $T_t(x, y)$ of $X_t$ satisfies~\eqref{eq:ttest} for some $r_0$. Assumption~\ref{assu:green} and scaling property~\itref{it:green} follow by Proposition~\ref{prop:green}. Since $X_t$ is symmetric (self-dual) and has continuous transition densities, Assumption~\ref{assu:dual} is also satisfied.

Finally, we assume that Assumption~\ref{assu:gen} holds with scaling property~\itref{it:gen} (see below). Under the above assumptions, scale-invariant boundary Harnack inequality holds with some localization radius. When $\X$ is unbounded, $\alpha \ne n$ and scaling property~\itref{it:reg} holds for all $r > 0$, then~\eqref{eq:ttest} holds for all $t > 0$ and all $x, y \in \X$, see~\cite{MR2008600}, and therefore we can take $R_0 = \infty$.

We list some cases when Assumption~\ref{assu:gen} with scaling property~\itref{it:gen} is known to hold true.
\begin{itemize}
\item When $\X = \R^k$ and $\nu(x, y)$ is a function of $x - y$, then $X_t$ is a symmetric L\'evy process and we can simply take $\dom = C_c^\infty(\R^k)$.
\item More generally, let $\X = \R^k$, and assume that $\nu(x, y) = \kappa(x, y) |y - x|^{-k - \alpha}$ for a $C_b^\infty(\R^k \times \R^k)$ function $\kappa$. We claim that Assumption~\ref{assu:gen} with scaling property~\itref{it:gen} holds for $\dom = C_c^\infty(\R^k)$. Indeed, for $f \in C_c^\infty(\R^k)$ let
\begin{align}
\label{eq:stablegen}
\begin{aligned}
 \tilde{\A} f(x) & = \int_{\R^k} \expr{f(x + z) - f(x) - \frac{z}{1 + |z|^2} \cdot \nabla f(x)} \frac{\kappa(x, x + z)}{|z|^{k + \alpha}} \, dz \\
 & \qquad + \expr{\int_{\R^k} \frac{z}{1 + |z|^2} \, \frac{\kappa(x, x + z) - \kappa(x, x)}{|z|^{k + \alpha}} \, dz} \cdot \nabla f(x) .
\end{aligned}
\end{align}
Then $\tilde{\A}$ is a symmetric pseudo-differential operator with appropriately smooth symbol, and by~\cite[Theorem~5.7]{MR1659620}, the closure of $\tilde{\A}$ is the Feller generator of a symmetric Hunt process $\tilde{X}_t$ (we omit the details). Since the pure-jump Feller processes $X_t$ and $\tilde{X}_t$ have equal L\'evy kernels, they are in fact equal processes, and hence the closure of $\tilde{\A}$ is the Feller generator of $X_t$. Assumption~\ref{assu:gen} with $\dom = C_c^\infty(\R^k)$ follows, and scaling property~\itref{it:gen} is a simple consequence of~\eqref{eq:stablegen}. See also~\cite{Hoh98, MR2718253}.
\item When $\alpha \in (0, 1)$, $\X$ is the closure of an open Lipschitz set, and $\nu(x, y) = c |x - y|^{-k - \alpha}$, then the desired condition is satisfied by $\dom = C_c^\infty(\R^k)$ (see~\cite[Theorem~6.1(i)]{MR2214908}).
\item When $\alpha \in [1, 2)$, $\X$ is the closure of an open set with $C^{1,\beta}$-smooth boundary for some $\beta > \alpha - 1$, and $\nu(x, y) = c |x - y|^{-k - \alpha}$, then one can take $\dom$ to be the class of $C_c^\infty(\R^k)$ functions with normal derivative vanishing everywhere on the boundary of $\X$ (see~\cite[Theorem~6.1(ii)]{MR2214908}).
\item For the case when $X_t$ is a subordinate diffusion on $\X$, see Example~\ref{ex:subordinate}. In this case, when $\X$ is a fractal set, one can even deal with $\alpha$ greater than $2$.
\end{itemize}
Note that an analytical proof of Theorem~\ref{th:regexit} discussed in Remark~\ref{rem:dirichlet} may lead to a generalization of this example, which would not require Assumption~\ref{assu:gen}.
\end{example}

\begin{example}[Stable-like subordinate diffusions in metric measure spaces]
\label{ex:subordinate}
Suppose that $(\X, d, m)$ is an Ahlfors regular $n$-space for some $n > 0$. Assume that the metric $d$ is uniformly equivalent to the shortest-path metric in $\X$. Suppose that there is a diffusion process $Z_t$ with a symmetric, continuous transition density $T^Z_t(x, y)$ satisfying the sub-Gaussian bounds
\begin{equation}
\label{eq:subgaussian}
\begin{aligned}
 \frac{c_1}{t^{n / d_w}} \, \exp \expr{-c_2 \expr{\frac{d(x, y)^{d_w}}{t}}^{1/ (d_w - 1)}} & \le T^Z_t(x, y) \\ & \hspace*{-7em} \le \frac{c_3}{t^{n / d_w}} \, 
 \exp \expr{-c_4 \expr{\frac{d(x, y)^{d_w}}{t}}^{1 / (d_w - 1)}}
\end{aligned}
\end{equation}
for all $x, y \in \X$ and $t \in (0, t_0)$ ($t_0 = \infty$ when $\X$ is unbounded). Here $d_w \ge 2$ is the \emph{walk dimension} of the space $\X$. The existence of such a diffusion process $Z_t$ is well-known when $\X$ is a Riemannian manifold ($d_w = 2$; see~\cite{MR2569498}), the $k$-dimensional Sierpi\'nski gasket ($d_w = \log (k+3) / \log 2 > 2$; see~\cite{MR966175}), more general nested fractals~\cite{MR1301625, MR1227032}, or the Sierpi\'nski carpets~\cite{MR1151799,MR1701339}; see~\cite{MR1840042} for more information.

Let $\alpha \in (0, d_w)$ and let $X_t$ be the stable-like process obtained by subordination of $Z_t$ with the $\alpha / d_w$-stable subordinator $\eta_t$, $X_t = Z(\eta_t)$. These processes were first studied in~\cite{MR2013738, MR2091704, MR1779007}. By the subordination formula, the transition density estimate~\eqref{eq:ttest} holds for some $r_0$ (if $\X$ is unbounded, then it was proved in~\cite{MR2013738} that we can take $r_0 = \infty$).

Since $X_t$ is symmetric and has continuous transition densities, Assumption~\ref{assu:dual} is clearly satisfied. The L{\'e}vy kernel of $X_t$ satisfies $c^{-1} d(x, y)^{-n - \alpha} \le \nu(x, y) \le c d(x, y)^{-n - \alpha}$, see~\cite{MR2013738}, and Assumption~\ref{assu:levy} with scaling property~\itref{it:nu3} follows. Assumption~\ref{assu:green} and scaling property~\itref{it:green} follow from the transition density estimate~\eqref{eq:ttest} by Proposition~\ref{prop:green}; see also~\cite[Lemmas~5.3 and~5.6]{MR2013738}. Finally,  Assumption~\ref{assu:gen} with scaling property~\itref{it:gen} follows by the construction of~\cite[Section~2]{MR2465816}. Roughly speaking, the method of~\cite{MR2465816} yields smooth bump functions in the domain of the generator of the diffusion $Z_t$ with appropriate scaling. By the subordination formula, these bump functions are in the domain of $\A$, and the constants scale appropriately. Since there are some nontrivial issues related to the construction, we repeat the construction with all details in Appendix~\ref{sec:app}. By Corollary~\ref{cor:smoothbump} there, Assumption~\ref{assu:gen} is satisfied with scaling property~\ref{it:gen}.

We conclude that scale-invariant boundary Harnack inequality for $X_t$ holds in the full range of $\alpha \in (0, d_w)$. Noteworthy, we obtain a regularity result also for $\alpha \ge 2$, when Lipshitz functions no longer belong to the domain of the Dirichlet form of $X_t$.

This example can be extended in various directions. Instead of taking $\eta_t$ the $\alpha / d_w$-stable subordinator, one can consider a subordinator $\eta_t$ whose Laplace exponent $\psi$ is a complete Bernstein function regularly varying of order $\alpha / d_w$ ($\alpha \in (0, d_w)$) at infinity. Such subordinators have no drift, and the L\'evy measure with completely monotone density function, regularly varying of order $-1 - \alpha/d_w$ at $0$. Their potential kernel is regularly varying of order $-1 + \alpha/d_w$ at $0$. We refer the reader to~\cite{MR2569321, Kim12potentialtheory, MR2598208} for more information about subordination, complete Bernstein functions and regular variation. By the subordination formula, following the method applied for the Euclidean case $\X = \R^k$ in~\cite{Kim12potentialtheory, MR2994122}, one can obtain two-sided estimates for the L\'evy kernel $\nu(x, y)$ and the potential kernel $\pot(x, y)$ in terms of $\psi$, at least when $\X$ is unbounded and $\alpha < d$. These estimates are sufficient to prove the scale-invariant boundary Harnack inequality.

Similar methods should be applicable also when $X_t$ is recurrent (that is, $\X$ is bounded, or $\alpha \ge d$). In this case, estimates of $\pot(x, y)$ need to be replaced by estimates of the $\lambda$-potential kernel $\pot_\lambda(x, y)$. Another interesting directions are the case of slowly varying $\psi$, which corresponds to $\alpha = 0$, and, on the other hand, the case of pure-jump processes with $\psi$ regularly varying of order $1$ (that is, $\alpha = d_w$). Finally, one can perturb processes considered above, in a similar way as in the next example.
\end{example}

\begin{example}[Stability under small perturbations]
\label{ex:perturbations}
Let $\X = \R^k$, $d$ be the Euclidean distance, $m$ be the Lebesgue measure, and $\alpha \in (0, 2)$. Suppose that $\tilde{\nu}(x, y)$ is a L\'evy kernel of a Hunt process $\tilde{X}_t$ considered in Example~\ref{ex:stable-like}, and $\tilde{\A}$ is the corresponding Feller generator. For example, $\tilde{\nu}(x, y)$ can be any function of $y - x$ satisfying~\eqref{eq:stablelevy}. In this example we consider a perturbation $\nu(x, y)$ of the kernel $\tilde{\nu}(x, y)$.

Although a more general construction is feasible, we are satisfied with the following setting. Let $\nu(x, y) = \tilde{\nu}(x, y) + n(x, y)$, where $n(x, y)$ is chosen so that $\nu(x, y)$ satisfies the scaling property~\itref{it:nu3}, $n(x, y)$ and $\hat{n}(x, y) = n(y, x)$ are kernels of bounded operators on $C_0(\R^k)$, and
\begin{align*}
 & \int_{\R^k} n(x, y) dy = \int_{\R^k} \hat{n}(x, y) dy , && x \in \R^k ;
\end{align*}
the last assumption guarantees that $m$ is an excessive (in fact, invariant) measure for the process $X_t$ defined below.

The formula $\mathcal{N} f(x) = \int_{\R^k} (f(y) - f(x)) n(x, y) dy$ defines a bounded linear operator on $C_0(\R^k)$, and $\A = \tilde{\A} + \mathcal{N}$ (defined on the domain of $\tilde{\A}$) has the positive maximum property. By a standard perturbation argument, $\A$ is the Feller generator of a Hunt process $X_t$, and $\nu(x, y)$ is the L\'evy kernel of $X_t$. The process $\hat{X}_t$ and its Feller generator $\hat{\A}$ are constructed in a similar manner, using the Feller generator of the dual of $\tilde{X}_t$ and the kernel $\hat{n}(x, y)$. It is easy to see that $\int_{\R^k} \A f(x) g(x) dx = \int_{\R^k} f(x) \hat{\A} g(x) dx$ for $f, g \in C_c^\infty(\R^k)$, from which it follows that $\hat{X}_t$ is indeed the dual of $X_t$.

The transition density of $\tilde{X}_t$ satisfies~\eqref{eq:ttest} (see Example~\ref{ex:stable-like}). The process $X_t$ can be constructed probabilistically using $\tilde{X}_t$ and Meyer's method of adding and removing jumps. Hence, by~\cite[Lemma~3.6]{MR2465826} and~\cite[Lemma~3.1(c)]{MR2492992}, the transition density of $X_t$ exists and also satisfies~\eqref{eq:ttest} for smaller $r_0$ (see also~\cite[Proposition~2.1]{MR2443765}).

It follows that Assumption~\ref{assu:dual} is satisfied. Assumption~\ref{assu:gen} holds with $\dom = C_c^\infty(\R^k)$, and scaling property~\itref{it:gen} (with finite $R_0$) follows from the $\alpha$-stable-like scaling of $\tilde{\A}$ and boundedness of $\mathcal{N}$. Since we assumed that~\itref{it:nu3} holds true, Assumption~\ref{assu:levy} is satisfied with scaling properties~\itref{it:nu}, \itref{it:nu2}. Assumption~\ref{assu:green} and scaling properties~\itref{it:tau}, \itref{it:green} follow from transition density estimate~\eqref{eq:ttest} by Proposition~\ref{prop:green}. Hence, scale-invariant boundary Harnack inequality holds true for $X_t$.

The above setting includes mixtures of isotropic stable processes (L{\'e}vy processes generated by $\A = -(-\Delta)^{\alpha/2} - c (-\Delta)^{\beta/2}$ with $0 < \beta < \alpha < 2$ and $c > 0$) and relativistic stable processes (L{\'e}vy processes generated by $\A = m - (-\Delta + m^{2/\alpha})^{\alpha / 2}$ with $m > 0$). Also, the dependence of constants on the parameters $c$, $\beta$, $m$ can be easily tracked. Since the perturbation $n(x, y)$ can be asymmetric, many non-symmetric processes are included. Finally, this example can be adapted to the setting of Ahlfors-regular $n$-sets in $\R^k$, as in Example~\ref{ex:stable-like}.
\end{example}

\begin{example}[Processes killed by a Sch{\"o}dinger potential]
\label{ex:schrodinger}
Suppose that the assumptions for the boundary Harnack inequality in Theorem~\ref{th:bhi} are satisfied. Let $\X'$ be an open set in $\X$. Let $M_t$ be a strong right-continuous multiplicative functional quasi-left continuous on $[0, \infty)$, for which all points of $\X'$ are permanent, and such that $M_t = 0$ for $t \ge \tau_{\X'}$. Finally, let $X^M_t$ be the subprocess corresponding to $M_t$ (in a similar way as in Section~\ref{sec:reg}; see~\cite{MR0264757} for definitions). Then $X^M_t$ is a Hunt process on $\X'$, uniquely determined by the relation $\pr^M_x(X^M_t \in E) = \ex_x(M_t; X_t \in E)$ for any $E \sub \X'$ and $x \in \X'$.

Assume that $M_t$ is a continuous function of $t \in [0, \tau_{\X'})$. We claim that in this case the L{\'e}vy kernel $\nu^M(x, y)$ of $X^M_t$ is again given by $\nu(x, y)$, restricted to $\X' \times \X'$. Indeed, by formula~\eqref{eq:psid} of Lemma~\ref{lem:psigen}, for $x \in \X'$ and $f \in \dom(\A)$ vanishing in a neighborhood of $x$, we have
\begin{align*}
 \ex^M_x(f(X^M_t)) - f(x) & = \ex_x(f(X_t) M_t) - f(x) \\
 & = \ex_x\expr{\int_0^t \A f(X_t) M_t dt} + \ex_x\expr{\int_0^t f(X_t) dM_t} .
\end{align*}
When divided by $t$, this converges (for a fixed $x$) to $\A f(x)$ as $t \to 0^+$. Hence, $\nu^M f(x) = \nu f(x)$. By an approximation argument, this holds for any $f \in C_c(\X')$ vanishing in a neighborhood of $x$, proving our claim. (Note that, however, in general, functions in $\dom(\A)$ need not belong to the domain of the generator of $X^M_t$, even if $\X' = \X$.)

We remark that many such functionals $M_t$ are related to Schr{\"o}dinger potentials $V$: for a nonnegative function $V$, we have $M_t = \exp(-\int_0^t V(X_s) ds)$ for $t < \tau_{\X'}$, see~\cite{MR0264757}. A similar construction was used in Section~\ref{sec:reg} for a particular choice of $V$. In some applications, the potential $V$ can take negative values, the case not covered by this example.

Let $D \sub \X$ be an open set. By the definition of a subharmonic function, a nonnegative function $f$ regular subharmonic on $D \cap \X'$ with respect to the process $X^M_t$, extended by $f(x) = 0$ for $x \in \X \setminus \X'$, is also regular subharmonic in $D$ with respect to $X_t$. Hence, the hypothesis of Theorem~\ref{th:regexit} holds for $X^M_t$ with the same constant $c_{\eqref{eq:regexit3}}$. Of course, one needs to replace the sets in the statement of Theorem~\ref{th:regexit} by their intersections with $\X'$.

We claim that also Lemma~\ref{lem:escape} holds for $X^M_t$ with the same constant. Indeed, with the definitions of the proof of Lemma~\ref{lem:escape} and $D' = D \cap \X'$, for $x \in B(x_0, r) \cap \X'$ we have
\begin{align*}
 \pr^M_x(X^M_{\tau_{D'}} \in \X' \setminus B(x_0, R)) & = \ex_x(M_{\tau_D}; X_{\tau_D} \in \X \setminus B(x_0, R)) \\
 & \le \ex_x(f(X_{\tau_D}) M_{\tau_D}) - f(x) .
\end{align*}
By formula~\eqref{eq:psid} of Lemma~\ref{lem:psigen},
\[
 \ex_x(f(X_{\tau_D}) M_{\tau_D}) - f(x) = \ex_x\expr{\int_0^{\tau_D} \A f(X_t) M_t dt} + \ex_x\expr{\int_{[0, \tau_D+]} f(X_t) d M_t} .
\]
The second summand on the right hand side is nonpositive. It follows that,
\begin{align*}
 \pr_x^M(X^M_{\tau_D} \in \X' \setminus B(x_0, R)) & \le \ex_x\expr{\int_0^{\tau_D} M_t dt} \sup_{y \in \X} \A f(y) \\
 & = \ex_x^M(\tau_{D'}) \sup_{y \in \X} \A f(y) ,
\end{align*}
as desired.

In Lemma~\ref{lem:factorization}, only the estimates of the L{\'e}vy measure and mean exit time are used. Therefore, also Lemma~\ref{lem:factorization} holds for the process $X^M_t$ with unaltered constants. In a similar way, the proof of Theorem~\ref{th:bhi} works for the process $X^M_t$ without modifications. We conclude that the boundary Harnack inequality holds for $X^M_t$ with the same constants. For convenience, we state this result as a separate theorem.

\begin{theorem}
\label{th:sub}
Suppose that Assumptions~\ref{assu:dual}, \ref{assu:gen}, \ref{assu:levy} and~\ref{assu:green} hold true. Let $\X'$ be an open subset of $\X$, and let $X^M_t$ be a subprocess of $X_t$, with state space $\X'$, corresponding to a strong right-continuous multiplicative functional for $X_t$, continuous before $X_t$ hits $\X \setminus \X'$, vanishing after that time, and quasi-left continuous on $[0, \infty)$. Then the boundary Harnack inequality holds true for the process $X^M_t$ with the same constant $c_{\eqref{eq:ubhi}}$ given by~\eqref{eq:ubhic1}. More precisely, if $x_0 \in \X$, $0 < r < R < R_0$, $D \sub B(x_0, R)$ is open, $f, g$ are nonnegative regular harmonic functions in $D \cap \X'$ (with respect to the process $X^M_t$), and $f, g$ vanish in $(B(x_0, R) \setminus D) \cap \X'$, then we have
\begin{align*}
 {f(x)}g(y) & \le c_{\eqref{eq:ubhi}} \, {g(x)}{f(y)} \, , && x, y \in B(x_0, r) \cap D \cap \X',
\end{align*}
where $c_{\eqref{eq:ubhi}} = c_{\eqref{eq:ubhi}}(x_0, r, R)$ does not depend on $M_t$.
\end{theorem}

We remark that the continuity assumption for $M_t$ is essential. If, for example, $M_t$ is equal to $1$ until the first jump larger than $1$, and then $0$, the boundary Harnack inequality typically does not hold, by an argument similar to one in Example~\ref{ex:truncated} below.
\end{example}

\begin{example}[Actively reflected and censored stable processes]
\label{ex:censored}
Let $\X' \sub \R^k$ be open and let $\X$ be the closure of $\X'$ in $\R^k$. Suppose that $\X$ satisfies property~\itref{it:reg}. Let $\nu(x, y) = c |x - y|^{-n - \alpha}$. As in Example~\ref{ex:stable-like}, under suitable assumptions on $\X$, there is a stable-like process $X_t$ with the L{\'e}vy kernel $\nu(x, y)$, and scale-invariant BHI holds for $X_t$. In~\cite{MR2006232}, the process $X_t$ is called actively reflected $\alpha$-stable process in $\X$, and the process $X'_t$, obtained from $X_t$ by killing it upon hitting $\X \setminus \X'$, is named censored $\alpha$-stable process in $\X'$ (see~\cite[Remark~2.1]{MR2006232}). Clearly, the boundary Harnack inequality for $X'_t$ is the special case of the boundary Harnack inequality for $X_t$, corresponding to open sets $D$ contained in $\X'$. (Note that this is in fact a special case of Theorem~\ref{th:sub}, with $M_t = 1$ for $t < \tau_{\X'}$.) Hence, we have scale-invariant BHI for the actively reflected $\alpha$-stable process $X_t$ and the censored $\alpha$-stable process $X'_t$, whenever $\X'$ is a Lipschitz set in the case $\alpha \in (0, 1)$, and $\X'$ is an open set with $C^{1,\beta}$-smooth boundary for some $\beta > \alpha - 1$ in the case $\alpha \in [1, 2)$. The above extends the results of~\cite{MR2006232, 2007arXiv0705.1614G}.
\end{example}

\begin{example}[Gradient-type perturbations of stable processes]
\label{ex:gradient}
Let $\alpha \in (1, 2)$. If $b : \R^k \to \R^k$ is bounded and differentiable, partial derivatives of $b$ are bounded, and $\diverg b = 0$, then the process $X_t$ generated by $-(-\Delta)^{\alpha/2} + b \cdot \nabla$, and the process $\hat{X}_t$ generated by $-(-\Delta)^{\alpha/2} - b \cdot \nabla$ are mutually dual. Such processes are considered in the recent paper~\cite{MR2875353}. The L{\'e}vy kernels of $X_t$ and $\hat{X}_t$ are the same as that of the isotropic $\alpha$-stable L{\'e}vy process generated by $(-\Delta)^{\alpha/2}$, see~\cite{springerlink:10.1007/s11118-011-9237-x}. Furthermore, $\dom = C_c^\infty(\R^k)$ is contained in the domains of $\A$ and $\hat{\A}$. Therefore, a scale-invariant (with finite $R_0$) boundary Harnack inequality holds for the process $X_t$.
\end{example}

We conclude this article with some negative or partially negative examples.

\begin{example}[L{\'e}vy processes with Brownian component]
\label{ex:brownian}
Let $\X = \R^k$, $d$ be the Euclidean distance, $m$ be the Lebesgue measure, and $\alpha \in (0, 2)$. Let $X_t$ be the sum of two independent processes, the Brownian motion and the isotropic $\alpha$-stable L{\'e}vy process. That is, $X_t$ is the L{\'e}vy process with generator $\A = c_1 \Delta - c_2 (-\Delta)^{\alpha/2}$.

Clearly, $X_t$ is symmetric and has transition densities, so Assumption~\ref{assu:dual} is satisfied. Furthermore, $\dom(\A)$ contains $C_c^\infty(\R^k)$, and hence Assumption~\ref{assu:gen} is satisfied with $2$-stable-like scaling: the property~\itref{it:gen} holds with $\alpha$ replaced by $2$. On the other hand, Assumption~\ref{assu:levy} clearly holds with $\alpha$-stable-like scaling~\itref{it:nu3}. Furthermore,  detailed estimates for the transition density of $X_t$ can be established (\cite{MR2677007}), from which Assumption~\ref{assu:green} follows as in Proposition~\ref{prop:green}, with $2$-stable scaling.

It follows that boundary Harnack inequality holds despite the diffusion component. However, the constant $c_{\eqref{eq:ubhi}}(x_0, r, R)$ is not bounded when, for example, $R = 2 r$ and $r \to 0^+$. This is a typical behavior for processes comprising both jump and diffusion part, and for general open sets one cannot expect a scale-invariant result: the boundary Harnack inequality in the form given in~\itref{it:bhi} does not hold for the Brownian motion without some regularity assumptions on the boundary of $D$, cf.~\cite{MR1127476}. On the other hand, the scale-invariant boundary Harnack inequality for $X_t$ in more smooth domains was established in~\cite{0908.1559}.
\end{example}

\begin{example}[Truncated stable processes]
\label{ex:truncated}
This example shows why Assumption~\ref{assu:levy} is essential for the boundary Harnack inequality in the form given in~\itref{it:bhi}. Consider the truncated isotropic $\alpha$-stable L{\'e}vy process $X_t$ in $\X = \R^k$, $\alpha \in (0, 2)$, $n \ge 1$. This is a pure-jump L{\'e}vy process with L{\'e}vy kernel $\nu(x, y) = c |x - y|^{-n - \alpha} \ind_{B(x, 1)}(y)$. Clearly, Assumptions~\ref{assu:dual}, \ref{assu:gen} and~\ref{assu:green}, as well as formula~\eqref{eq:tau}, hold true with $\alpha$-stable-like scaling and $R_0 = 1$, but Assumption~\ref{assu:levy} is violated.

We examine two specific harmonic functions. Let $v$ be a vector in $\R^d$ with $|v| = 2/3$, let $r \in (0, 1/6)$ be a small number, and define $B_1 = B(x_1, r)$ and $B_2 = B(x_2, r)$, where $x_1, x_2 \in \R^k$ are arbitrary points satisfying $x_1 - x_2 = v$. Let $D = B_1 \cup B_2$, $E_1 = B_1 + v$, $E_2 = B_2 - v$, and let $f_j(x) = \pr_x(X(\tau_D) \in E_j)$. Suppose that $x \in B_1$. By~\eqref{eq:iw}, we have
\begin{align*}
 3^{-n-\alpha} c |E_1| \ex_x \tau_{B_1} & \le f_1(x) \le 3^{n + \alpha} c |E_1| \ex_x \tau_D .
\end{align*}
When $x \in B_2$, then, again by~\eqref{eq:iw},
\begin{align*}
 f_1(x) & \le \pr_x(X(\tau_{B_2}) \in B_1) \cdot \sup_{y \in B_1} f_1(y) \\
 & \le c 3^{n + \alpha} |B_1| \ex_x \tau_{B_2} \cdot 3^{n + \alpha} c |E_1| \sup_{y \in B_2} \ex_y \tau_D .
\end{align*}
Similar estimates hold true for $f_2$. It follows that
\begin{align*}
 \frac{f_1(x_2) f_2(x_1)}{f_1(x_1) f_2(x_2)} & \le c_{n,\alpha}^2 (3^{n + \alpha})^6 |B_1| \, |B_2| \expr{\sup_{y \in D} \ex_y \tau_D}^2 \\
 & \le c_{n,\alpha}^2 (3^{n + \alpha})^6 |B(0, 1)|^2 r^{2n} \expr{\sup_{y \in B(0, 1)} \ex_y \tau_{B(0, 1)}}^2 .
\end{align*}
This ratio can be arbitrarily small when $r \to 0$, and therefore~\itref{it:bhi} cannot hold for truncated stable process uniformly with respect to the domain. We remark that by an appropriate modification of the above example, one can even construct a single domain (an infinite union of balls) for which~\itref{it:bhi} is false. Also, modifications of the above example for other truncated processes, or for processes with super-exponential decay of the density of the L{\'e}vy measure can be given.

On the other hand, if the regular harmonic functions $f$ and $g$ (of the truncated $\alpha$-stable process $X_t$) vanish outside a unit ball, then clearly $f$ and $g$ are harmonic in $D$ also with respect to the standard (that is, non-truncated) isotropic $\alpha$-stable process in $\R^k$. Therefore, the boundary Harnack inequality actually holds true for such functions. A different version of boundary Harnack inequality was proved for $X_t$ under some regularity assumptions on the domain of harmonicity in~\cite{MR2282263, MR2391246}.
\end{example}

%
%

\appendix

\section{Smooth bump functions on metric measure spaces with sub-Gaussian heat kernels}
\label{sec:app}

In this part we repeat the construction of smooth bump functions of~\cite{MR2465816}. We adopt the setting of Example~\ref{ex:subordinate}: $Z_t$ is a diffusion process on an Ahlfors-regular $n$-space $\X$, the transition semigroup $T^Z_t$ of $Z_t$ satisfies sub-Gaussian bounds~\eqref{eq:subgaussian}, and $X_t$ is defined to be the process $Z_t$ subordinated by an independent $\alpha/d_w$-stable subordinator, $\alpha \in (0, d_w)$. The generator of $Z_t$ serves as the (Neumann) Laplacian $\Delta$ on $\X$, and $T^Z_t$ is the heat semigroup.

Let $h = T^Z_t g$ for some $t > 0$ and $g \in L^2(\X)$. One of the main results of~\cite{MR2465816}, Theorem~2.2, states that given any compact $K$ and $\eps > 0$, there is a function $f$ such that $f \in \dom(\Delta^l)$ for all $l > 0$, $f(x) = h(x)$ on $K$ and $f(x) = 0$ when $\dist(x, K) \ge \eps$. There are at least three issues when one tries to apply this result in our setting.

First, Theorem~2.2 in~\cite{MR2465816} is given under the assumption that the spectral gap of $\Delta$ is positive. However, this assumption is used only in the proof of Lemma~2.6, which contains a flaw: positivity of the spectral gap $\lambda$ does not imply the inequality $\|P_t f - f\|_{L^2(\X)} \le \lambda t \|f\|_{L^2(\X)}$ (see line~3 on page~1769 and line~12 on page~1773 in~\cite{MR2465816}). This issue has been resolved by the authors of~\cite{MR2465816} in an unpublished note, containing a corrected version of the proof of Lemma~2.6. The new argument does not involve the condition on the spectral gap, which therefore turns out to be superfluous. For future reference, we provide the corrected version of the proof of Lemma~2.6 below.

Second, to get Assumption~\ref{assu:gen}, we need to apply the above theorem with $h(x) = 1$ for $x \in K$, where $h = T^Z_t g$. This condition is satisfied when $g(x) = 1$ for all $x \in \X$. However, such a function $g$ is in $L^2(\X)$ only when $m$ is a finite measure, and the general case is not covered by~\cite{MR2465816}. For that reason, we choose to repeat the construction of~\cite{MR2465816} in $L^\infty(\X)$ (instead of $L^2(\X)$) setting. 

Finally, for a scale-invariant boundary Harnack inequality, we need an upper bound for $\|\Delta f\|_{L^\infty(\X)}$ with explicit dependence on scale, that is, explicit in $\eps$ and the size (e.g. the diameter) of $K$. Such properties of the estimates are irrelevant in~\cite{MR2465816}, but it turns out that they can be obtained by carefully following the proof of Theorem~2.2 in~\cite{MR2465816}.

For the above reasons, we decide to give a complete proof of an $L^\infty(\X)$ version of Theorem~2.2 in~\cite{MR2465816}. However, it should be emphasized that method was completely developed in~\cite{MR2465816}. Although we only need the result for $g(x) = h(x) = 1$ for all $x \in \X$, for future reference we consider the general case. 

\begin{theorem}[{a variant of~\cite[Theorem~2.2]{MR2465816}}]
\label{th:smoothbump}
Suppose that $K \sub \X$ is a compact set, $\eps, s > 0$ and $h = T^Z_s g$ for some $g \in L^\infty(\X)$. Then there is a function $f \in L^\infty(\X)$ such that $f(x) = h(x)$ for $x \in K$, $f(x) = 0$ when $\dist(x, K) \ge \eps$, and $f \in \dom(\Delta^l)$ for any $l > 0$. Furthermore, the $L^\infty(\X)$ norm of $f$ is bounded by the $L^\infty(\X)$ norm of $g$, $f$ is nonnegative if $g$ is nonnegative, and for all $l > 0$ we have
\begin{equation}
\label{eq:smoothbump}
 \norm{\Delta^l f}_{L^\infty(\X)} \le \frac{c_{\eqref{eq:smoothbump}} (\diam K + \eps)^{n/2}}{\eps^{l d_w + n/2}} \, \norm{g}_{L^\infty(\X)} ,
\end{equation}
where $c_{\eqref{eq:smoothbump}} = c_{\eqref{eq:smoothbump}}(l, \eps^{d_w} / s, Z_t)$.
\end{theorem}

\begin{proof}
We divide the argument into five steps. All constants in this proof may depend not only on the parameters given in parentheses, but also on the space $\X$ and the process $Z_t$. Since we never refer to the semigroup of the subordinate process $X_t$, in this proof for simplicity we write $T_t = T^Z_t$. Furthermore, also in this proof only, we extend $\Delta$ to the $L^\infty(\X)$ generator of $T_t$ (recall that originally $\Delta$ was defined as the $C_0(\X)$ generator), and denote by $\Delta_{L^2(\X)}$ the $L^2(\X)$ generator of $T_t$, that is, the generator of the semigroup of operators $T_t$ acting on $L^2(\X)$. Clearly, $\Delta f = \Delta_{L^2(\X)} f$ $m\aevery$ whenever $f \in \dom(\Delta) \cap \dom(\Delta_{L^2(\X)})$.

\textit{Step 1.} We begin with some general estimates. By the spectral theorem and the inequality $\lambda^l e^{-\lambda t} \le (l e / t)^l$, for any $f \in L^2(\X)$ and $l \ge 0$, we have $T_t f \in \dom((\Delta_{L^2(\X)})^l)$, and
\begin{align*}
 \norm{(\Delta_{L^2(\X)})^l T_t f}_{L^2(\X)} & \le (l e / t)^l \norm{f}_{L^2(\X)} .
\end{align*}
Furthermore, by sub-Gaussian estimates~\eqref{eq:subgaussian}, $\|T_t(x, \cdot)\|_{L^2(\X)} = (T_{2t}(x, x))^{1/2} \le c_1 t^{-n / (2 d_w)}$. Hence,
\begin{align*}
 \norm{T_t f}_{L^\infty(\X)} & \le c_1 t^{-n / (2 d_w)} \norm{f}_{L^2(\X)} .
\end{align*}
We find that
\begin{align*}
 & \norm{\frac{T_s T_t f - T_t f}{s} - \Delta_{L^2(\X)} T_t f}_{L^\infty(\X)} \\
 & \qquad \le \frac{c_1}{(t/2)^{n / (2 d_w)}} \norm{\frac{T_s T_{t/2} f - T_{t/2} f}{s} - \Delta_{L^2(\X)} T_{t/2} f}_{L^2(\X)} \to 0
\end{align*}
as $s \to 0^+$. It follows that $T_t f \in \dom(\Delta)$, with $\Delta T_t f = \Delta_{L^2(\X)} T_t f$. By a similar argument, $T_t f \in \dom(\Delta^l)$ for any $l \ge 0$, and
\begin{align*}
 \norm{\Delta^l T_t f}_{L^\infty(\X)} & = \norm{T_{t/2} \Delta^l T_{t/2} f}_{L^\infty(\X)} \le \frac{c_2}{t^{n / (2 d_w)}} \, \norm{\Delta^l T_{t/2} f}_{L^2(\X)} \le \frac{c_3(l)}{t^{l + n / (2 d_w)}} \norm{f}_{L^2(\X)} .
\end{align*}
Sub-Gaussian estimate~\eqref{eq:subgaussian} and Ahlfors regularity of $\X$ also give the following estimate: for any set $E \sub \X$, any $\eps > 0$ and any $f \in L^\infty(\X)$ or $f \in L^1(\X)$ vanishing in the $\eps$-neighborhood of $E$, we have
\begin{align*}
 \norm{T_t f}_{L^\infty(E)} & \le D(\eps, t) \norm{f}_{L^\infty(\X)} , && \text{and} & \norm{T_t f}_{L^1(E)} & \le D(\eps, t) \norm{f}_{L^1(\X)} ,
\end{align*}
where
\begin{align*}
 D(\eps, t) & = \sup_{x \in \X} \int_{\X \setminus B(x, \eps)} T_t(x, y) m(dy) \le c_4 \exp(-c_5 (\eps^{d_w} / t)^{1 / (d_w - 1)}) .
\end{align*}
In particular, given any $s > 0$ and $\eps > 0$ it is possible to choose a strictly increasing sequence $s_j > 0$ convergent to $s$, with $s_0 = 0$, such that if $t_j = s_j - s_{j-1}$ ($j \ge 1$), then
\begin{align*}
 \lim_{j \to \infty} D(2^{-j} \eps, s - s_j) & = 0 && \text{and} & \sum_{i = 1}^\infty \frac{D(2^{-i} \eps, t_i)}{t_{i+1}^l} & \le \frac{c_6(l, \eps^{d_w} / s)}{\eps^{l d_w}} < \infty
\end{align*}
for any $\eps > 0$, $l \ge 0$. For example, one can take $s_j = (1 - 4^{-d_w j}) s$. Note, however, that the above series would diverge if $t_j$ decreased either too slowly or too rapidly.

\textit{Step 2.} Let $g \in L^\infty(\X)$, $\eps, s > 0$, and $h(x) = T_s g(x)$, as in the statement of the theorem. Following~\cite{MR2465816}, for $j \ge 0$ we define
\begin{align*}
 K_j & = \{ x \in \X : \dist(x, K) < 2^{-j} \eps \} , & L_j & = \{ x \in \X : \dist(x, K) > (1 - 2^{-j}) \eps \} ,
\end{align*}
and $A_j = \X \setminus (K_j \cup L_j)$. Furthermore, let $s_j$ and $t_j$ be chosen as in Step~1. For $j \ge 1$ we define
\begin{align*}
 u_0(x) & = 0 , & u_j(x) & = \ind_{K_j}(x) T_{s_j} g(x) + \ind_{A_j}(x) T_{t_j} u_{j-1}(x) .
\end{align*}
Below we prove that $T_{s - s_j} u_j$ converges to a function $f$ with the desired properties.

\textit{Step 3.}
By induction, $\|u_j\|_{L^\infty(\X)} \le \|g\|_{L^\infty(\X)}$ for any $j \ge 0$. For $j \ge 1$ we have $u_{j-1}(x) = T_{s_{j-1}} g(x)$ for $x \in K_{j-1}$, and $\dist(K_j, \X \setminus K_{j-1}) \ge 2^{-j} \eps$. Hence,
\begin{equation}
\label{eq:kbound}
\begin{aligned}
 \norm{u_j - T_{t_j} u_{j-1}}_{L^\infty(K_j)} & = \norm{T_{t_j} (T_{s_{j-1}} g - u_{j-1})}_{L^\infty(K_j)} \\
 & \hspace*{-5em} \le D(2^{-j} \eps, t_j) \norm{T_{s_{j-1}} g - u_{j-1}}_{L^\infty(\X)} \le 2 D(2^{-j} \eps, t_j) \norm{g}_{L^\infty(\X)} ,
\end{aligned}
\end{equation}
where $D(2^{-j} \eps, t_j)$ is as in Step~1 (cf.~\cite[Lemma~2.3]{MR2465816}). Also, $u_j$ vanishes on $L_j$, $u_{j-1}$ vanishes on $L_{j-1}$, and $\dist(L_j, \X \setminus L_{j-1}) \ge 2^{-j} \eps$, so that
\begin{align*}
 \norm{u_j - T_{t_j} u_{j-1}}_{L^\infty(L_j)} & = \norm{T_{t_j} u_{j-1}}_{L^\infty(L_j)} \le D(2^{-j} \eps, t_j) \norm{u_{j-1}}_{L^\infty(\X)} ,
\end{align*}
and (using $\X \setminus L_{j-1} \sub K_0$)
\begin{align*}
 \norm{u_j - T_{t_j} u_{j-1}}_{L^1(L_j)} & = \norm{T_{t_j} u_{j-1}}_{L^1(L_j)} \\
 & \le D(2^{-j} \eps, t_j) \norm{u_{j-1}}_{L^1(\X)} \le D(2^{-j} \eps, t_j) m(K_0) \norm{u_{j-1}}_{L^\infty(\X)} .
\end{align*}
Hence, using also $\|u_{j-1}\|_{L^\infty(\X)} \le \|g\|_{L^\infty(\X)}$, we obtain (cf.~\cite[Lemma~2.5]{MR2465816})
\begin{equation}
\label{eq:lbound}
\begin{aligned}
 \norm{u_j - T_{t_j} u_{j-1}}_{L^2(L_j)} & \le \sqrt{\norm{u_j - T_{t_j} u_{j-1}}_{L^\infty(L_j)} \norm{u_j - T_{t_j} u_{j-1}}_{L^1(L_j)}} \\
 & \le D(2^{-j} \eps, t_j) \sqrt{m(K_0)} \, \norm{g}_{L^\infty(\X)} .
\end{aligned}
\end{equation}

\textit{Step 4.}
We follow the corrected version of the proof of~\cite[Lemma~2.6]{MR2465816}. Let $l \ge 0$. For $j \ge 1$ we have
\begin{align*}
 \Delta^l T_{s - s_j} u_j & = \sum_{i = 1}^j \Delta^l T_{s - s_i} (u_i - T_{t_i} u_{i-1}) .
\end{align*}
Observe that the results of Step~1 and the equality $u_i(x) = T_{t_i} u_{i-1}(x)$ for $x \in A_i$ give
\begin{align*}
 \sum_{i = 1}^\infty \norm{\Delta^l T_{s - s_i} (u_i - T_{t_i} u_{i-1})}_{L^\infty(\X)} & \le \sum_{i = 1}^\infty \frac{c_3(l)}{(s - s_i)^{l + n / (2 d_w)}} \norm{u_i - T_{t_i} u_{i-1}}_{L^2(\X)} \\
 & \hspace*{-11em} \le \sum_{i = 1}^\infty \frac{c_3(l)}{t_{i+1}^{l + n / (2 d_w)}} \expr{\norm{u_i - T_{t_i} u_{i-1}}_{L^2(K_i)} + \norm{u_i - T_{t_i} u_{i-1}}_{L^2(L_i)}} .
\end{align*}
Hence, by~\eqref{eq:kbound} and~\eqref{eq:lbound},
\begin{align*}
 \sum_{i = 1}^\infty \norm{\Delta^l T_{s - s_i} (u_i - T_{t_i} u_{i-1})}_{L^\infty(\X)} & \le 3 c_3(l) \sqrt{m(K_0)} \, \norm{g}_{L^\infty(\X)} \sum_{i = 1}^\infty \frac{D(2^{-i} \eps, t_i)}{t_{i+1}^{l + n / (2 d_w)}} \\
 & \le 3 c_3(l) \sqrt{m(K_0)} \, \norm{g}_{L^\infty(\X)} \, \frac{c_6(l + n / (2 d_w), \eps^{d_w} / s)}{\eps^{l d_w + n/2}} \, .
\end{align*}
It follows that the sequence $\Delta^l T_{s - s_j} u_j$ converges in $L^\infty(\X)$ as $j \to \infty$ for every $l \ge 0$. Therefore, if $f(x) = \lim_{j \to \infty} T_{s - s_j} u_j(x)$, then for all $l \ge 0$ we have $f \in \dom(\Delta^l)$ and
\begin{align*}
 \norm{\Delta^l f}_{L^\infty(\X)} & \le \sum_{i = 1}^\infty \norm{\Delta^l T_{s - s_i} (u_i - T_{t_i} u_{i-1})}_{L^\infty(\X)} \\
 & \le \frac{c_7(l, \eps^{d_w} / s)}{\eps^{l d_w + n/2}} \, \sqrt{m(K_0)} \, \norm{g}_{L^\infty(\X)} ,
\end{align*}
as desired.

\textit{Step 5.} By the definition of $u_j$, for $j \ge 1$ we have
\begin{align*}
 T_{s - s_j} u_j & = T_{s - s_j} (\ind_{K_j} T_{s_j} g + \ind_{A_j} T_{t_j} u_{j-1}) \\
 & = T_s g + T_{s - s_j} (\ind_{A_j} T_{t_j} u_{j-1} - \ind_{\X \setminus K_j} T_{s_j} g) .
\end{align*}
It follows that
\begin{align*}
 \norm{T_{s - s_j} u_j - T_s g}_{L^\infty(K)} & = \norm{T_{s - s_j} (\ind_{A_j} T_{t_j} u_{j-1} - \ind_{\X \setminus K_j} T_{s_j} g)}_{L^\infty(K)} \\
 & \le D(2^{-j} \eps, s - s_j) \norm{\ind_{A_j} T_{t_j} u_{j-1} - \ind_{\X \setminus K_j} T_{s_j} g}_{L^\infty(\X)} \\
 & \le 2 D(2^{-j} \eps, s - s_j) \norm{g}_{L^\infty(\X)} .
\end{align*}
The right hand side converges to $0$ as $j \to \infty$. Hence, $f(x) = T_s g(x) = h(x)$ for $x \in K$. Furthermore, $\|u_j\|_{L^\infty(\X)} \le \|g\|_{L^\infty(\X)}$, and therefore also $\|f\|_{L^\infty(\X)} \le \|g\|_{L^\infty(\X)}$. Finally, if $g \ge 0$, then $u_j \ge 0$ for all $j \ge 1$, and so $f \ge 0$.
\end{proof}

By choosing $g(x) = h(x) = 1$ and $s = \eps^{d_w}$, we obtain the following result.

\begin{corollary}
\label{cor:smoothcutoff}
Suppose that $K \sub \X$ is a compact set and $\eps > 0$. Then there is a function $f \in L^\infty(\X)$ such that $f(x) = 1$ for $x \in K$, $f(x) = 0$ when $\dist(x, K) \ge \eps$, and $f \in \dom(\Delta^l)$ for any $l > 0$. Furthermore, $0 \le f(x) \le 1$ for all $x \in \X$, and for all $l > 0$ we have
\begin{equation}
\label{eq:smoothcutoff}
 \norm{\Delta^l f}_{L^\infty(\X)} \le \frac{c_{\eqref{eq:smoothcutoff}} (\diam K + \eps)^{n/2}}{\eps^{l d_w + n/2}} \, ,
\end{equation}
where $c_{\eqref{eq:smoothcutoff}} = c_{\eqref{eq:smoothcutoff}}(l, Z_t)$.
\qed
\end{corollary}

In general, the boundary of the set $\{x \in \X : f(x) > 0\}$ might be highly irregular. However, when we relax the smoothness hypothesis on $f$, we can require $f$ to be positive on an arbitrary given open set.

\begin{proposition}
\label{prop:smoothcutoff}
Suppose that $K \sub \X$ is a compact set, $\eps > 0$ and $L > 0$. Then there is a function $f \in L^\infty(\X)$ such that $f(x) = 1$ for $x \in K$, $f(x) = 0$ when $\dist(x, K) \ge \eps$, and $f \in \dom(\Delta^l)$ for $l = 1, 2, \ldots, L$. Furthermore, $0 \le f(x) \le 1$ for all $x \in \X$, the boundary of the set $\{x \in \X : f(x) > 0\}$ has zero $m$ measure, and for all $l = 1, 2, \ldots, L$ we have
\begin{equation}
\label{eq:smoothcutoff2}
 \norm{\Delta^l f}_{L^\infty(\X)} \le \frac{c_{\eqref{eq:smoothcutoff2}} (\diam K + \eps)^{n/2}}{\eps^{l d_w + n/2}} \, ,
\end{equation}
where $c_{\eqref{eq:smoothcutoff2}} = c_{\eqref{eq:smoothcutoff2}}(L, Z_t)$.
\end{proposition}

\begin{proof}
Let $f_0$ be the function constructed in Theorem~\ref{th:smoothbump} for $h(x) = g(x) = 1$, and denote by $V$ an arbitrary open set with the following properties: $\{x \in \X : f(x) > 0\} \sub V \sub \{x \in \X : \dist(x, K) < 2\eps\}$, and $m(\partial V) = 0$. For example, one can take $V = \{x \in \X : \dist(x, K) < r\}$ for a suitable $r \in (\eps, 2 \eps)$.

Let $B_j$, $j = 1, 2, \ldots$, be a family of balls contained in $V \cap \{x \in \X : f_0(x) < 1/2\}$ such that twice smaller balls $B_j'$ form a countable covering of $V \cap \{x \in \X : f_0(x) < 1/2\}$, and let $f_j$ be the function as in Corollary~\ref{cor:smoothcutoff}, equal to $1$ on $B_j'$ and vanishing on $\X \setminus B_j$. Finally, choose $\eps_j > 0$ so that for $l = 0, 1, \ldots, L$,
\begin{align*}
 \sum_{i = 1}^\infty \eps_i \norm{\Delta^l f_i}_{L^\infty(\X)} < \frac{1}{2} \, \norm{\Delta^l f_0}_{L^\infty(\X)} .
\end{align*}
Then $f = f_0 + \sum_{i = 1}^\infty \eps_i f_i$ has all the desired properties, with $\eps$ replaced by $2 \eps$.
\end{proof}

\begin{corollary}
\label{cor:smoothbump}
Assumption~\ref{assu:gen} holds with $\alpha$-stable scaling.
\end{corollary}

\begin{proof}
Given any compact subset $K$ of an open set $D \sub \X$, choose $\eps > 0$ such that $\dist(\X \setminus D, K) \ge \eps$. Since $\dom(\Delta) \sub \dom(\A)$, the function $f$ given in Proposition~\ref{prop:smoothcutoff} (for $L = 1$) satisfies all conditions of Assumption~\ref{assu:gen}. Furthermore, if $\nu_\eta(s) ds$ is the L{\'e}vy measure of the subordinator $\eta_t$, then
\begin{align*}
 \norm{\A f}_{L^\infty(\X)} & = \norm{\int_0^\infty (T^Z_s f - f) \nu_\eta(s) ds}_{L^\infty(\X)} \le \int_0^\infty \norm{T^Z_s f - f}_{L^\infty(\X)} \nu_\eta(s) ds \\
 & \le \int_0^\infty \min\expr{s \norm{\Delta f}_{L^\infty(\X)}, 2 \norm{f}_{L^\infty(\X)}} \nu_\eta(s) ds .
\end{align*}
Note that $\|f\|_{L^\infty(\X)} = 1$. Furthermore, $\min(\lambda s, 2) \le c_1 (1 - e^{-\lambda s})$ (with $c_1 = 2 e^2 / (e^2 - 1)$), and
\begin{align*}
 \int_0^\infty (1 - e^{-\lambda s}) \nu_\eta(s) ds = \lambda^{\alpha / d_w} .
\end{align*}
Therefore,
\begin{align*}
 \norm{\A f}_{L^\infty(\X)} & \le c_1 \expr{\norm{\Delta f}_{L^\infty(\X)}}^{\alpha / d_w} .
\end{align*}
Let $0 < r < R$, and take $K = \overline{B}(x_0, r)$, $D = B(x_0, R)$, $\eps = R - r$. We see that
\begin{align*}
 \norm{\A f}_{L^\infty(\X)} & \le c_1 \expr{\frac{c_{\eqref{eq:smoothcutoff2}}(1, Z_t) (2 R)^{n/2}}{(R - r)^{d_w + n/2}}}^{\alpha / d_w} = c_2(r/R, Z_t) R^{-\alpha} .
\end{align*}
This gives half of the $\alpha$-stable scaling property~\ref{it:gen}, and the other half is proved in a similar manner.
\end{proof}

%
%

\subsection*{Acknowledgments}

We express our gratitude to Moritz Ka{\ss}mann for many discussions on the subject of supremum bounds for subharmonic functions. We thank Tomasz Grzywny for pointing out errors in the preliminary version of the article. We also thank Luke Rogers for providing us with a new proof of Theorem~2.2 in~\cite{MR2465816}. We thank the referee for insightful suggestions.

%
%

\def\cprime{$'$}

%
%

\end{document}